\documentclass[11pt,a4paper,thmsb,ukenglish]{article}
\usepackage{epsf,  amsmath, amssymb, graphicx, enumerate}
\usepackage{amsthm, chngcntr}
\usepackage[sort]{natbib}
\usepackage{a4wide}
\usepackage{subfigure}
\usepackage{xcolor}
\usepackage{pdflscape}
\usepackage{hyperref}
\usepackage{bbm}
\usepackage{xr}
\usepackage{accents}
    
\input{math_commands.tex}

\definecolor{darkred}{rgb}{0.8,0,0}
\newcommand{\myRed}[1]{\textcolor{black}{#1}}

\newcommand{\A}{\ensuremath{{\bf A}}} 
\newcommand{\Kt}{\ensuremath{ \bf{K}}} 
\newcommand{\Iit}{\ensuremath{\mathbb I}} 
\newcommand{\Mt}{\ensuremath{\mathbb M}} 
\newcommand{\An}{\ensuremath{\overline{\bf A}}} 
\newcommand{\Id}{\ensuremath{{\bf I}_{d \times d}}} 
\newcommand{\Isquare}{\ensuremath{{\bf I}_{d^2 \times d^2}}} 
\newcommand{\Q}{\ensuremath{\textbf{Q}}} 
\newcommand{\MdR}{\ensuremath{{\mathcal{M}_{d}(\R)}}} 
\newcommand{\MndR}{\ensuremath{{\mathcal{M}_{n,d}(\R)}}} 
\newcommand{\MdB}{\ensuremath{{\mathcal{M}_{d}(\{0,1\})}}} 
\newcommand{\vectorise}{\ensuremath{\mathrm{vec}}}

\newcommand{\N}{\ensuremath{\mathbb{N}}}

\newcommand{\AL}{\ensuremath{\textnormal{AL}}}
\newcommand{\supp}{\ensuremath{\textnormal{supp}}}

\newcommand{\diag}{\ensuremath{\mathrm{diag}}}
\newcommand{\interior}{\ensuremath{\mathrm{int}}}

\newcommand{\trace}{\ensuremath{\mathrm{tr}}}

\newcommand{\indicator}{\ensuremath{\mathbb{I}}}
\newcommand{\proba}{\ensuremath{{\text{P}}}} 
\newcommand{\determinant}{\ensuremath{{\text{det}}}} 
\newcommand{\Acov}{\ensuremath{{\text{Acov}}}} 

\newcommand{\YY}{\ensuremath{\mathbb{Y}}} 
\newcommand{\HH}{\ensuremath{\mathbb{H}}} 
\newcommand{\XX}{\ensuremath{\mathbb{X}}} 
\newcommand{\ZZ}{\ensuremath{\mathbb{Z}}} 
\newcommand{\ZZZ}{\ensuremath{\boldsymbol{Z}}} 
\newcommand{\UUU}{\ensuremath{\boldsymbol{U}}} 
\newcommand{\JJ}{\ensuremath{\mathbb{J}}} 
\newcommand{\Ss}{\ensuremath{\mathbb{S}_d}} 

\newcommand{\Si}{\ensuremath{ {\boldsymbol \Sigma}}}

\newcommand{\boldTheta}{\ensuremath{ {\boldsymbol{\theta}}}}
\newcommand{\boldPsi}{\ensuremath{ {\boldsymbol{\psi}}}}
\newcommand{\boldB}{\ensuremath{ {\boldsymbol{b}}}}

\newcommand{\LL}{\ensuremath{\mathbb{L}}}
\newcommand{\LLL}{\ensuremath{\boldsymbol{L}}}
\newcommand{\VVV}{\ensuremath{\boldsymbol{V}}}
\newcommand{\WW}{\ensuremath{\mathbb{W}}}

\DeclareFontFamily{U}{mathx}{\hyphenchar\font45}
\DeclareFontShape{U}{mathx}{m}{n}{
      <5> <6> <7> <8> <9> <10>
      <10.95> <12> <14.4> <17.28> <20.74> <24.88>
      mathx10
      }{}
\DeclareSymbolFont{mathx}{U}{mathx}{m}{n}
\DeclareFontSubstitution{U}{mathx}{m}{n}

\newcommand{\Levy}{L\'{e}vy\ }
\newcommand{\levy}{L\'{e}vy}

\newtheorem{theorem}{Theorem}

\numberwithin{theorem}{subsection}

\newtheorem{assumption}{Assumption}

\newtheorem{corollary}[theorem]{Corollary}

\newtheorem{definition}{Definition}

\newtheorem{lemma}[theorem]{Lemma}
\newtheorem{notation}[theorem]{Notation}

\newtheorem{proposition}[theorem]{Proposition}
\newtheorem{remark}[theorem]{Remark}

\allowdisplaybreaks

\begin{document}
\title{Likelihood theory for the Graph Ornstein-Uhlenbeck process}
\author{\textsc{Valentin Courgeau\footnote{Corresponding author. Email: valentin.courgeau15@imperial.ac.uk}   \quad    Almut E.\ D.\ Veraart} \\
\textit{Department of Mathematics, Imperial College London}\\
\textit{ 180 Queen's Gate, 
 London, SW7 2AZ, 
UK}  \\
}
\maketitle
\begin{abstract}
We consider the problem of modelling restricted interactions between continuously-observed time series as given by a known static graph (or network) structure. \myRed{For this purpose, we define a parametric multivariate Graph Ornstein-Uhlenbeck (GrOU) process driven by a general \Levy process to study the momentum and network effects amongst nodes, effects that quantify the impact of a node on itself and that of its neighbours, respectively. We derive the maximum likelihood estimators (MLEs) and their usual properties (existence, uniqueness and efficiency) along with their asymptotic normality and consistency. Additionally, an Adaptive Lasso approach, or a penalised likelihood scheme, infers both the graph structure along with the GrOU parameters concurrently and is shown to satisfy similar properties.} Finally, we show that the asymptotic theory extends to the case when stochastic volatility modulation of the driving \Levy process is considered. 
\end{abstract}

 \noindent{\bf Keywords:} 
Ornstein-Uhlenbeck processes, multivariate \Levy process, continuous-time likelihood, maximum likelihood estimator, graphical modelling, central limit theorem, adaptive Lasso.\\


\section{Introduction}

\myRed{Ornstein-Uhlenbeck (OU) models, driven by Brownian motion or \Levy processes, form a class of continuous-time models with a broad range of applications: in finance for pairs trading \citep{hol2018estimation, endres2019tradingLevyDrivenOU} and volatility modelling \citep{barndorff2002bnsmodel, pigorsch2009DefinitionSemiPositiveMultOU}, in neuroscience \citep{Melanson2019DataDrivenStationaryJumpDiffusions}, or even in electricity management \citep{Longoria2019OULevyElectricityPortfolios}. In parallel, high-dimensional time series datasets fostered the development of sparse inference for OU-type processes \citep{Boninsegna2018SparseLearningDynamical, Gaiffas2019SparseOUProcess} as a way to control interactions within complex systems. On the other hand, graphical time series models are usually restricted to discrete-time time series \citep{zhu2017network, knight2016modelling, knight2019generalised}.} 

\myRed{From this observation, we introduce a mean-reverting process formulated as a graphical model and we derive its theoretical properties. This model contributes to the development of methodologies leveraging both the modelling flexibility of continuous-time models and the sparsity of graphical models. Furthermore, this article unlocks both hypothesis testing and uncertainty estimation in this setting.}

\myRed{In our framework, a graph is a set of nodes that interact together through a collection of links, or edges. Each node is described by a numerical value at any time and we observe a multivariate time series, where each component is the value of one node at a given time. To accommodate this graph structure, we introduce the \emph{Graph Ornstein-Uhlenbeck (GrOU)} process, a parametric autoregressive model placed on a graph and driven by a \Levy process \citep{Masuda2004, masuda2007ergodicity}. This strictly stationary model quantifies the impact of a node value on its increments, the \emph{momentum} effect, and the impact of the neighbouring nodes, the \emph{network} effect. It is interpreted as a continuous-time extension to the discrete-time Network Vector Autoregression (NAR) model \citep{zhu2017network}---a model used for the modelling of social media \citep{zhu2020multivariate}, financial returns \citep{chen2020community} or risk dynamics \citep{Chen2019TailEventDrivenNetworks}. We also introduce an alternative formulation with separate momentum and network effects for each node for a finer modelling of the graph dynamics.}

\myRed{For more details, the general multivariate \levy-driven OU process and its ergodic properties are presented in \cite{Masuda2004, masuda2007ergodicity, Sandric2016, Kevei2018}. We refer to \cite{sorensen1991likelihood} for a general theory of the likelihood inference for continuously-observed jump diffusions.}

\myRed{This article focuses on providing the theoretical guarantees necessary for applications.
First, we introduce the \levy-driven GrOU model. Although the estimation of \levy-driven OU processes have been largely considered \citep{masuda2010, Gushchin2020DriftEstimation}, to the best of the authors' knowledge, the formulation of the GrOU process is the first of its kind.}

\myRed{Second, under non-degeneracy conditions, we show that this model belongs to an exponential family \citep{kuchler1997exponentialbook, mancini2009non} with a fully-explicit likelihood function. We derive corresponding maximum likelihood estimators (MLEs) to estimate the momentum and networks effects. Those estimators are shown to be consistent and efficient in the sense of H{\'a}jek-Le Cam's convolution theorem \citep{hajek1970efficiencyLan, LeCam1990}. Those results extend the continuous-time asymptotic results given in the \levy-driven univariate setting \citep{mai2014efficient} and the multivariate case with a Brownian noise \citep{Hopfner2014AsymptoticStatistics, basak2008asymptotic}. }

\myRed{Then, we prove that this family of models is locally asymptotically normal (LAN) \citep{LeCam1990} and provide a couple of central limit theorems (CLTs) under square integrability conditions. In addition, we extend the adaptive Lasso regularisation scheme from \cite{Gaiffas2019SparseOUProcess} to the \levy-driven case and prove the asymptotic normality of the estimator. This approach can be used when the graph structure is not known a priori: our approach can not only be applied on graphical time series but also on general multivariate time series with sparse dependencies.}

\myRed{Finally, by assuming that the covariance matrix of the \Levy process is itself a matrix-valued OU process, i.e.\ both time-dependent and stochastic, the GrOU model is equipped with a stochastic covariance modulation term \citep{pigorsch2009DefinitionSemiPositiveMultOU} along with a pure-jump term. We show that the resulting process would still be ergodic using the mixing criterion from \cite{fuchs2013mixing}. The asymptotic results are shown to hold conditional on the knowledge of this volatility process and open the doors to a more complex modelling of graphical interactions of time series.}


The GrOU process is presented in Section \ref{section:network-ou} where we describe formally the momentum and network effects along with non-degeneracy conditions. Then, Section \ref{section:likelihood-and-estimators} is devoted to setting up the likelihood framework and to proving the existence and uniqueness of the MLEs. In Section \ref{section:asymptotic-theory-continuous-times}, both the LAN property of the model along with the H{\'a}jek-Le Cam efficiency of the MLEs are provided. We also present the MLE CLTs with an explicit limiting covariance matrix. \myRed{We present an Adaptive Lasso scheme as well as extend known Brownian asymptotic properties to the \levy-driven case in Section \ref{section:asymptotics-adaptive-lasso}.} Finally, in Section \ref{section:stoch-vol}, we extend further our framework to include an OU-type stochastic volatility term and we show that, conditional on the knowledge of the stochastic volatility process, the central limit theorems from Sections \ref{section:asymptotic-theory-continuous-times} \& \ref{section:asymptotics-adaptive-lasso} hold.
   
\section{A Graph Ornstein-Uhlenbeck process}
\label{section:network-ou}
In this section, we define the Graph Ornstein-Uhlenbeck (GrOU) process \myRed{to study either the aggregated or individual behaviour of nodes on the graph.}

\subsection{Notations}
\label{sectionL:notations}
We consider a filtered probability space $(\Omega, \mathcal{F}, (\mathcal{F}_t,\ t \in \R), \proba_0)$ to which all stochastic processes are adapted. We consider two-sided \Levy processes $(\LL_t,\ t \in \R)$  (i.e stochastic processes with stationary and independent increments and continuous in probability and $\LL_0 = \boldsymbol{0}_{\rmD}, \ \proba_0-a.s.$)  \citep[Remark 1]{Brockwell2009} which are without loss of generality assumed to be c{\`a}dl{\`a}g and we write $\YY_{t-}:= \lim_{s \uparrow t} \YY_s$ for any $t \in \R$. For any probability measure $\proba$, we denote by $\proba_t$ its restriction to the $\sigma$-field $\mathcal{F}_t$ for any $t \in \R$.

We denote by $\determinant$ the matrix determinant, the space of $\{0,1\}$-valued $d \times d$ matrices by $\MdB$, the space of real-valued $d \times d$ (resp.\ $n\times d$) matrices by $\MdR$ (resp.\ \MndR), the linear subspace of $d\times d$ symmetric matrices by $\Ss$, the (closed in $\Ss$) positive semidefinite cone (i.e.\ with the real parts of their eigenvalues non-negative) by $\Ss^+$ and the (open in $\Ss$) positive definite cone (i.e.\ with the real parts of their eigenvalues positive) by $\Ss^{++}$. In particular, $\Id \in \MdR$ denotes the $d\times d$ identity matrix.

We denote by $\lambda^{leb}$ the one-dimensional Lebesgue measure. For a non-empty topological space, $\mathcal{B}(S)$ is the Borel $\sigma$-algebra on $S$ and $\pi$ is some probability measure on $(S,\mathcal{B}(S))$. The collection of all Borel sets in $S \times \R$ with finite $\pi \otimes \lambda^{leb}$-measure is written as $\mathcal{B}_b(S \times \R)$. Also, the norms of vectors and matrices are denoted by $\|\cdot \|$. We usually take the Euclidean (or Frobenius) norm but due to the equivalence between norms, our results are not norm-specific and are valid under any norm in $\R^d$ or $\MdR$.
In addition, for an invertible matrix $\boldsymbol{M} \in \MdR$, we define $\langle \boldsymbol{x}, \boldsymbol{y} \rangle_{\boldsymbol{M}} := \boldsymbol{x}^\top \boldsymbol{M^{-1}} \boldsymbol{y}$ for $\boldsymbol{x}, \boldsymbol{y} \in \R^d$. Finally, for a process $\XX_t= (X^{(1)}_t,\dots,X^{(d)}_t)^\top\in\R^d$ we denote by $[\XX]_t$ the matrix $([X^{(i)},X^{(j)}]_t)$ of quadratic co-variations up to time $t \geq 0$.

 In this article, $\otimes$ denotes the Kronecker matrix product, $\odot$ is for the Hadamard (element-wise) matrix product and $\vectorise$ is the vectorisation transformation where columns are stacked on one another. We denote the inverse vectorisation transformation by $\vectorise^{-1}(\boldsymbol{x}) := (\vectorise(\Id)^\top \otimes \Id) \cdot (\Id \otimes \boldsymbol{x}) \in \MdR$ for $\boldsymbol{x}\in \R^{d^2}$.

\subsection{The \levy-driven Ornstein-Uhlenbeck process}
\label{section:levy-noise-details}
We recall the construction of Ornstein-Uhlenbeck (OU) processes and introduce their graphical interpretation, namely the GrOU process. We give two parametrisations specific to the modelling of graph structures namely the $\boldTheta$-GrOU and the $\boldPsi$-GrOU: the former encodes the momentum and network effects for the whole graph into two scalar parameters while the latter has separate parameters for each node and each neighbour.

We consider a $d$-dimensional OU process $\YY_t = (Y^{(1)}_t,\dots,Y^{(d)}_t)^\top$ for $t \geq 0$ satisfying the stochastic differential equation (SDE) for a \emph{dynamics} matrix $\rmQ\in\Ss^{++}$
\begin{equation}
	\label{eq:sde}
	d\YY_t = -\rmQ \YY_{t-}dt + d\LL_t,
\end{equation}
for a two-sided $d$-dimensional \Levy process $\LL_t=(L_t^{(1)},\ldots,L_t^{(d)})^{\top}$ \citep[Remark 1]{Brockwell2009} such that $\YY_0$ is independent of $(\LL_t, \ t \geq 0)$ \citep[Section 1]{Masuda2004}.

The \Levy process $\LL$ is defined by the \levy-Khintchine characteristic triplet $(\boldB, \Si, \nu)$ with respect to the truncation function $\tau(\boldsymbol{z}) := \mathbb{I}_{\{\boldsymbol{x} \in \R^d : \|\boldsymbol{x}\| \leq 1\}}(\boldsymbol{z})$ where $\mathbb{I}$ denotes the indicator function. More explicitly, the \levy-Khintchine representation yields, for $t \in \R$,
 $$\small{\E\left[\exp\left(i \boldsymbol{u}^\top \LL_t \right)\right] = \exp\left\{
t\left( i\boldsymbol{u}^\top \boldB - \frac{1}{2}\boldsymbol{u}^\top \Si \boldsymbol{u}+ \int_{\R^d\symbol{92}\{\boldsymbol{0}_d\}}\left[ \exp\left(i \boldsymbol{u}^\top \boldsymbol{z} \right) - 1 - i \boldsymbol{u}^\top \boldsymbol{z} \tau(\boldsymbol{z})\right]d\nu(\boldsymbol{z})\right)\right\}},$$ where $\boldsymbol{u},\ \boldB \in \R^d$, $\Si \in \Ss^{++}$ and $\nu$ is a \Levy measure on $\R^d$ satisfying $\int_{\R^d\symbol{92}\{\boldsymbol{0}\}} (1 \wedge \|\boldsymbol{z}\|^2)\nu(d\boldsymbol{z}) < \infty$. 

Again, without loss of generality, consider that $\proba_0$ is a probability measure where $(\LL_t,\ t \in \R)$ is a c{\`a}dl{\`a}g \Levy process as mentioned above. Also, we denote by $\proba_{t,0}$ the probability measure $P_{0}$ restricted to the $\sigma$-field $\mathcal{F}_t$ as introduced in Section \ref{sectionL:notations}.

\subsection{The OU process on a graph}
	The components of $\YY$ are interpreted as the \emph{nodes} of a graph structure linked together through a collection of \emph{edges}. Those are given in a \emph{adjacency} (or graph topology) matrix  $\A=(a_{ij})\in \MdB$: : $a_{ij}=1$ for an existing link between node $i$ and $j$, 0 otherwise. We assume that $a_{ii}=0$ for all $i \in \{1,\ldots,d\}$. Requiring the knowledge of the adjacency matrix is necessary for the graphical interpretation of the OU process. This limitation is alleviated in the sparse inference scheme introduced in Section \ref{section:asymptotics-adaptive-lasso}. 
\begin{assumption}
\label{assumption:A-known}
We assume that $\A$ is deterministic, static in time and known.
\end{assumption}
For $i\in \{1,\ldots,d\}$, define $n_i:= 1 \vee \sum_{j\not = i} a_{ij}$, which counts the number of neighbouring nodes the $i$-th node is connected to, or node degree. We can now \myRed{define} the row-normalised adjacency matrix
\begin{align*}
\An:= \diag(n_1^{-1},\ldots,n_d^{-1})\A.
\end{align*}
\myRed{to normalise the parameter values representing the momentum and networks effects which we introduce below.}
\paragraph{The $\boldTheta$-GrOU process}
We introduce the two-dimensional parameter vector $\boldTheta := (\theta_1, \theta_2)^\top \in \R^2$ describing the network effect and the momentum effect respectively, and define the matrix 
\begin{equation}
\label{eq:Q-using-theta}
\rmQ  = \Q(\boldTheta) :=\theta_2 \Id + \theta_1 \An.
\end{equation}
\myRed{The row-normalised adjacency matrix $\An$ makes the momentum parameter $\theta_2$ and the network parameter $\theta_1$ in \eqref{eq:Q-using-theta} directly comparable with one another independently of the node degrees.} This yields the SDE given by
$$d\YY_t = -\Q(\boldTheta) \YY_{t-} dt + d\LL_t,$$
whose $i$-th component satisfies the equation 
$$dY^{(i)}_t = - \theta_2 Y^{(i)}_tdt - \theta_1 n_i^{-1}\sum_{j\neq i} a_{ij}Y^{(j)}_t dt + dL^{(i)}_t,  \ t \geq 0.$$
\paragraph{The $\boldPsi$-GrOU process} Recall that we denote by $\odot$ the Hadamard product and by $\vectorise^{-1}$ the inverse vectorisation transformation. In general, for a $d^2$-dimensional vector $\boldPsi \in \R^{d^2}$, we have
\begin{equation}
\label{eq:Q-using-psi}
\rmQ = \Q(\boldPsi) := (\Id + \An) \odot \vectorise^{-1}(\boldPsi),
\end{equation}
\myRed{where
$$ \vectorise^{-1}(\boldPsi) = 
\begin{pmatrix}
\psi_{1} \ &		\dots 	& \psi_{d(d-1)+1} \\
\vdots   \ &  	\ddots 	& \vdots \\
\psi_{d} \ &  	\dots 	& \psi_{d\times d} \\
\end{pmatrix},$$}
such that the corresponding SDE is written as:
$$dY^{(i)}_t = -Q_{ii} Y^{(i)}_tdt - \sum_{j \neq i} Q_{ij} Y^{(j)}_t dt + dL^{(i)}_t, \quad t \geq 0.$$
\myRed{Formally, we obtain
$$dY^{(i)}_t = -\underbrace{\boldPsi_{d(i-1)+i}}_{\text{momentum effect}}Y^{(i)}_tdt - \underbrace{n_i^{-1}\sum_{j \neq i} a_{ij}\boldPsi_{d(j-1)+i}}_{\text{network effect}}Y^{(j)}_tdt + dL^{(i)}_t, \quad t \geq 0.$$}

 This second parametrisation alleviates the scarcity of network interactions imposed by $\boldTheta$ yet exposes the estimation to the curse of dimensionality as the number of nodes $d$ grows. For simplicity, one may write $\Q$ for $\Q(\boldTheta)$ or $\Q(\boldPsi)$ when the context is clear. 
With $\Q(\boldTheta)$, we restrict the interactions to the network and momentum effects. This extends the current framework to partially observable networks (e.g.\ too large for computations) where an exploration process needs to take place \citep[Section 5]{DereichMorters2013randomnetworks}. This makes the estimation robust again the curse of dimensionality coming from the number of nodes. We now define the GrOU process as follows.
\begin{definition}
	\label{definition:grou}
	The Graph Ornstein-Uhlenbeck (GrOU) process is a c{\`a}dl{\`a}g process $(\YY_t,\ t \geq 0)$ satisfying Equation \eqref{eq:sde} for some two-sided \Levy noise $(\LL_t,\ t \in \R)$ where $\Q$ is given by either Equation \eqref{eq:Q-using-theta} or by Equation \eqref{eq:Q-using-psi} such that $\Q$ is positive definite. This process is then called a $\boldTheta$-GrOU process or a $\boldPsi$-GrOU process, respectively.
\end{definition}
We give sufficient conditions for $\Q$ to be positive definite in both cases in Section \ref{section:stationary-solution}.

\subsection{Stationary solution}
\label{section:stationary-solution}
Recall that $\Ss^{++}$ is the set of $d \times d$ matrices such that the real parts of the eigenvalues are positive. \myRed{We restrict ourselves to study strictly stationary time series, we give the known conditions under which the OU process is strictly stationary.}
\begin{remark}
Note that $\rmQ \in \Ss^{++}$ if and only if $\determinant(e^{-t\Q}) \xrightarrow{t \rightarrow + \infty} 0$ \citep{Masuda2004}.
\end{remark}

Given the standard OU processes theory from \cite{Masuda2004,Brockwell2007}, we assume the following:
\begin{assumption}
\label{assumption:strong-solution-vector}
Suppose that $\rmQ \in \Ss^{++}$ and that the \Levy measure $\nu(\cdot)$ satisfies the log moment condition:
\begin{equation*}
    \int_{\| \boldsymbol{z} \|>1}\ln \|\boldsymbol{z}\|\nu(d\boldsymbol{z}) < \infty.
\end{equation*}
\end{assumption}

Then, under Assumption \ref{assumption:strong-solution-vector}, there is a unique strictly stationary solution of the above SDE given by
\begin{equation}
\label{eq:stationary-solution-vec}
\YY_t = e^{-(t-s)\Q}\YY_s + \int_s^t e^{-(t-u)\Q}d\LL_u, \qquad \text{for any $t \geq s \geq 0$.}
\end{equation}
Recall that the \levy-Khintchine characteristic triplet of $\LL$ with respect to the truncation function $\tau$ is denoted $(\boldB, \Si, \nu)$. Proposition 2.1, \cite{Masuda2004} yields that the transition probability from $\boldsymbol{x}$ at time $t$ denoted $\proba(t,\boldsymbol{x},\cdot)$ is characterised by the triplet $(b_{t,\boldsymbol{x}}, C_t, \nu_t)$ (with respect to $\tau$) defined as
\begin{align*}
	b_{t,\boldsymbol{x}} &:= e^{-t\Q}x + \int_0^t e^{-s\Q}bds + \int_{\R^d}\int_0^t e^{-s\Q}z\left[\tau(e^{-s\Q}z) - \tau(z) \right]ds\nu(dz),\\
	C_t &:= \int_0^t e^{-s\Q}Ce^{-s\Q^\top}ds, \quad \nu_t(S) := \int_0^t \nu\left(e^{s\Q}S\right)ds, \quad \text{for any } S \in \mathcal{B} (\R^d),
\end{align*}
where the limit as $t \rightarrow \infty$ leads a characteristic triplet of the form $(\boldB_{\infty}, C_{\infty}, \nu_{\infty})$ which characterises the unique invariant distribution of $\YY$ denoted by $\pi$ i.e.\ $\YY_t \xrightarrow{\ \mathcal{D} \ } \YY_{\infty} \sim \pi $ as $t \rightarrow \infty$. 

\begin{proposition}
\label{proposition:nar-equiv-identifiability}
If $\theta_2 > 0$ such that $\theta_2 > |\theta_1|$, then $\Q(\boldTheta) \in \Ss^{++}$.
\end{proposition}
\begin{proof}
Recall that the Ger\v{s}gorin's circle theorem states that any eigenvalue of $\Q(\boldTheta)$ is found in a closed circle of centre $Q_{ii} = \theta_2$  and radius equal to the sum of non-diagonal entries of the $i$-th row $\sum_{j\neq i}|Q_{ij}|$ for some $i\in\{1,\dots,d\}$. 
Since $\An$ is row-normalised and given that $\theta_2 > |\theta_1|$, we conclude that the eigenvalues must be in an open disk with center $\theta_2$ and radius $|\theta_1|$. This disk is positioned in the positive half-plane and \myRed{does} not contain the origin. This yields that all eigenvalues are strictly positive.
%
\end{proof}

Proposition \ref{proposition:nar-equiv-identifiability} makes practical sense as it requires that the auto-regressive part of the model is predominant in absolute terms. For the $\boldPsi$-GrOU formulation, we obtain the following:
\begin{proposition}
\label{proposition:node-level-psi-identiability}
If we have
$$ \psi_{i(d-1) +i} > 0 \quad \text{and} \quad \psi_{d(i-1)+i} > n_i^{-1}\sum_{j\neq i}|\psi_{d(j-1)+i}| \quad \text{for any $i \in \{1,\dots,d\}$,}$$
then $\Q(\boldPsi) \in \Ss^{++}$. This means that $\vectorise^{-1}(\boldPsi) \in \MdR$ has positive diagonal elements.
\end{proposition}
\begin{proof}
	The absolute sum of the off-diagonal elements of the $i$-th row of $\Q(\boldPsi)$ give 
	$$\sum_{j \neq i}|Q_{ij}| = n_{i}^{-1}\sum_{j \neq i} a_{ij}|\psi_{d(j-1)+i}| \leq  n_{i}^{-1}\sum_{j \neq i} |\psi_{d(j-1)+i}| < |\psi_{d(i-1)+i}|.$$
	Similarly to the proof of Proposition \ref{proposition:nar-equiv-identifiability}, applying Ger\v{s}gorin's circle theorem yields the expected result.
\end{proof}
\begin{remark}
\label{remark:var}
	The model defined in \eqref{eq:sde} is a generalisation of the (discrete-time) Vector Autoregressive (VAR) model \citep{sims1980macroeconomics} with a depth of 1. \myRed{For a step size $\Delta > 0$, we then denote the sampled process by $\XX_j:=(X_j^{(1)},\ldots, X_j^{(d)})^{\top}$ for $j\in \N$, where $X_{j}^{(k)}=Y_{j \Delta}^{(k)}$. Hence we have the \textnormal{VAR(1)}-representation
\begin{align*}
\XX_j = {\boldsymbol \Phi} \XX_{j-1} + \ZZ_j,\quad j\in \N,
\end{align*}
where the parameter matrix ${\boldsymbol \Phi}$ is given by
\begin{align*}
{\boldsymbol \Phi}=e^{-\Delta\Q}=e^{-\Delta(\theta_1 \An + \theta_2 \Id)}=e^{-\Delta\theta_1 \An}e^{-\Delta\theta_2 \Id},
\end{align*} 
and an i.i.d.\ noise sequence given by
\begin{align*}
\ZZ_j =\int_{(j-1)\Delta}^{j\Delta} e^{-(j\Delta-u)\Q}d\LL_u.
\end{align*}}
\end{remark}
\myRed{This VAR-like formulation with a fixed step size $\Delta$ is studied in \cite{fasen2013} where one estimates $e^{-\Delta \Q}$ as a $d \times d$ matrix directly. However, a strong identifiability issue hinders the estimation of $\Q$ from $e^{-\Delta \Q}$: the logarithm of a matrix (or log-matrix) is not necessarily itself a real matrix and, if so, it may be not unique \citep[p.\ 1146]{culver1966logmatrix}. The former holds if and only if $e^{-\Delta \Q}$ is nonsingular and each elementary divisor (Jordan block) of $e^{-\Delta \Q}$ belonging to a negative eigenvalue occurs an even number of times \citep[Th.\ 1]{culver1966logmatrix}. Since $e^{-\Delta \Q}$ has only positive eigenvalues, a real-valued log-matrix always exist. On the other hand, the uniqueness is more difficult to ascertain. It requires that all the eigenvalues of $e^{-\Delta \Q}$ are positive and real and that no elementary divisor (or Jordan block) of $e^{-\Delta \Q}$ belonging to any eigenvalue appears more than once \citep[Th.\ 2]{culver1966logmatrix}. Those conditions are difficult to check in practice: the eigenvalues $\lambda$ of $\Q$ have positive real parts but may have non-zero imaginary parts such that $e^{-\Delta\lambda}$ may be non-real. Also, the Jordan block assumption does not hold in general for all $\Q \in \Ss^{++}$. With this issue in mind, we introduce a fully-explicit likelihood function  in Section \ref{section:likelihood-and-estimators}; whose logarithm is well-defined as a real-valued positive function.} 

\section{Likelihood and estimators}
\label{section:likelihood-and-estimators}
Suppose we observe the process $\YY$ continuously on $[0,T]$ for $T\in \R\cup\{\infty\}$ and let $t \in [0,T]$. In this section, we present the likelihood framework of interest and derive closed-form formulas for the $\boldTheta$-GrOU and $\boldPsi$-GrOU MLEs. 
\subsection{Ergodicity}
Recall the ergodic theorem of a process $(\YY_t,\ t \in \R)$ satisfying Equation \eqref{eq:sde} is given by:
 \begin{proposition}{(Theorems 2.1 \& 2.6, \cite{masuda2007ergodicity})\\}
\label{proposition:ergodicity-vector}
Suppose that Assumptions \ref{assumption:A-known} and \ref{assumption:strong-solution-vector} hold. Then $(\YY_t, \ t \in \R)$ admits a unique invariant distribution $\pi$ for any choice of the law $\eta$ of the initial value $\YY_0$. Moreover, for any measurable function $g:\R^d \mapsto \R^{k}$ satisfying $\mathbb{E}\left[\|g(\YY_\infty)\|\right] := \int_{\R^d}\|g(y)\|\pi(dy) < \infty$ and for any $\eta$, we have 

\begin{equation}
    \label{eq:ergodicity-vector}
    \frac{1}{t}\int_0^t g(\YY_s)ds \xrightarrow{t \rightarrow \infty} \mathbb{E}\left[g(\YY_\infty)\right] := \int_{\R^d}g(y)\pi(dy), \qquad \proba_\eta-a.s.
\end{equation}
where $\proba_\eta$ is the law of $\YY$ associated with the initial value $\YY_0 \sim \eta$.
\end{proposition}
According to Section \ref{section:stationary-solution}, recall that $\YY_t \xrightarrow{\ \mathcal{D} \ }\YY_\infty \sim \pi$ as $t \rightarrow \infty$ where $\pi$ is the invariant distribution characterised by $(\boldB_\infty, C_\infty, \nu_\infty)$.
This result is essential for the asymptotic properties of the model statistical inference and was further extended in \cite{Sandric2016, Kevei2018}.
We present the general likelihood framework used for GrOU processes.

\subsection{The fully-explicit likelihood function}
To infer the model parameters, we set up an explicit likelihood function to be maximised. For a general positive definite dynamics matrix $\Q$, the Radon-Nikodym derivative of the corresponding $d$-dimensional Ornstein-Uhlenbeck process (Eq.\ \ref{eq:sde}) 
is expressed as follows \citep{pap1996parameter, mai2014efficient}:
\begin{equation}
    \label{eq:radon-nikodym-derivative}
    \frac{dP_{t,\YY}}{d\proba_{t,0}} = \exp\left\{-\int_0^t \langle \rmQ \YY_{s}, d\YY^c_s\rangle_{\Si} - \frac{1}{2}\int_0^t\langle \rmQ \YY_s, \rmQ \YY_s\rangle_{\Si} ds \right\}, \quad t \in [0,T],
\end{equation}
where $\YY^c_s$ is the continuous $\proba_{0}$-martingale part and $\proba_{\YY}$ is a probability measure equivalent to $\proba_{0}$ such that the process in Equation \eqref{eq:radon-nikodym-derivative} is a martingale. Similarly to $\proba_{t,0}$, $\proba_{t,\YY}$ is the restriction of $\proba_{\YY}$ on $\mathcal{F}_t$ for any $t \in \R$.  We note that, since $\Si \in \Ss^{++}$, $\Si$ is invertible and hence Equation \eqref{eq:radon-nikodym-derivative} is well-defined. 

\begin{remark}
	In the rest of the article, we write $t=T$ (i.e.\ $t$ is the time horizon itself) and consider the likelihood at time $t$ directly.
\end{remark}

\subsection{The case of the $\theta$-GrOU process}
In this section, we write the likelihood for $\boldTheta$-GrOU processes and obtain the corresponding maximum likelihood estimator along with its existence and uniqueness. Consider the notations:
\begin{notation}
\label{notation:H-and-C}
Define the deterministic positive definite matrix
$$\boldsymbol{G}_\infty :=
\begin{pmatrix}
\E\left( \langle \An \YY_\infty, \An \YY_\infty \rangle_{\Si} \right) & \E\left(  \langle \An \YY_\infty,  \YY_\infty \rangle_{\Si} \right)\\
\E\left( \langle \An \YY_\infty,  \YY_\infty \rangle_{\Si} \right) & \E\left( \langle  \YY_\infty,  \YY_\infty \rangle_{\Si} \right)
\end{pmatrix},$$
and define
$$
\HH_t := - \begin{pmatrix}
\int_{0}^t\langle \An\YY_{s}, d\YY^c_s \rangle_\Si \\
\int_{0}^t\langle \YY_{s}, d\YY^c_s \rangle_\Si
\end{pmatrix} \ \text{such that} \ [\HH]_t = \begin{pmatrix}
		\int_0^t\langle \An \YY_{s},\An\YY_{s} \rangle_\Si ds & \int_0^t\langle \An\YY_{s},\YY_{s} \rangle_\Si ds\\
		 \int_0^t\langle \An\YY_{s},\YY_{s} \rangle_\Si ds & \int_0^t\langle \YY_{s},\YY_{s} \rangle_\Si ds
	\end{pmatrix}
.$$
\end{notation}
Using Equation \eqref{eq:Q-using-theta}, we deduce that
\begin{equation}
\label{eq:laplace-theta-likelihood}
	\int_0^t \langle \Q(\boldTheta) \YY_s, \Q(\boldTheta) \YY_s\rangle_{\Si} ds = \boldsymbol{\theta}^\top \cdot 
	[\HH]_t \cdot \boldsymbol{\theta}.
\end{equation}

\begin{lemma}
\label{lemma:HH_t-limit}
Suppose that Assumptions \ref{assumption:A-known} and \ref{assumption:strong-solution-vector} hold and that $\YY$ has finite second moments. Then, we have that $t^{-1}[\HH]_t$ converges almost surely to 
$\boldsymbol{G}_\infty$ as $t \rightarrow \infty$. Therefore, $[\HH]_t=O(t)$ componentwise as $t\rightarrow\infty$ and we obtain that $\E\left(|[\HH]_t|\right) \rightarrow \infty$ componentwise as $t \rightarrow \infty$.	
\end{lemma}
\begin{proof}
By stationarity and ergodicity, $t^{-1}[\HH]_t$ converges almost surely to $\boldsymbol{G}_\infty$ as $t \rightarrow \infty$. This matrix has finite elements since $\YY$ has finite second moments. By Jensen's inequality, we obtain $\E\left(|\HH_t|\right) \geq |\E\left([\HH]_t\right)|$. By Fubini's theorem and stationarity we have that $\E\left([\HH]_t\right) = \int_0^t \boldsymbol{G}_\infty dt = O(t)$ componentwise as $t \rightarrow \infty$ and hence $\E\left(|[\HH]_t|\right) \rightarrow \infty$ componentwise as $t \rightarrow \infty$.
\end{proof}

We then set up a likelihood framework (in the sense of Section 2, \cite{morales2000renyi}) for the $\boldTheta$-GrOU process as follows:
\begin{proposition}
\label{proposition:radon-nikodym}
From Equation \eqref{eq:radon-nikodym-derivative}, we obtain the likelihood ratio $\mathcal{L}_t(\boldsymbol{\theta}; \YY)$:
$$
\mathcal{L}_t(\boldsymbol{\theta}; \YY) := \frac{dP_{t,\YY}}{d\proba_{t,0}} = \exp \left(\boldsymbol{\theta}^{\top}\HH_t - \frac{1}{2}  \boldsymbol{\theta}^\top \cdot [\HH]_t \cdot  \boldsymbol{\theta}\right), \quad \boldTheta \in \R^2.  
$$
Additionally, $t \mapsto \HH_t$ and $t \mapsto \boldsymbol{\theta}^\top \cdot [\HH]_t \cdot  \boldsymbol{\theta}$ (for a fixed $\boldsymbol{\theta})$ are c\`adl\`ag hence bounded and $\mathcal{F}_t$-measurable for any finite $t \in R$.
\end{proposition}
\begin{proof}
\myRed{According to Lemma \ref{lemma:HH_t-limit} and by continuity, $\HH$ and $[\HH]$ are finite for any $t \geq t_0$ where $t_0$ is such that both $\HH_{t_0}$ and $[\HH]_{t_0}$ are finite. Also, $t \mapsto [\HH]_t$ has a linear growth as $t \rightarrow \infty$.} Therefore, the likelihood is defined on $\boldsymbol{\Theta} := \{\boldsymbol{\theta} \in \R^2 : |\boldsymbol{\theta}^\top \cdot [\HH]_t \cdot  \boldsymbol{\theta}| < \infty, \ \forall t \geq 0 \} = \R^2$. By the properties of the stochastic integral with respect to the continuous martingale part of $\YY$, $t \mapsto \HH_t$ is indeed c\`adl\`ag.
\end{proof}

From Proposition \ref{proposition:nar-equiv-identifiability}, we define a compact set $\widehat{\boldsymbol{\Theta}}$ such that
$$\widehat{\boldsymbol{\Theta}} \subseteq \{(\theta_1, \theta_2)^\top \in \R^2: \theta_2 > 0 \text{ and } \theta_2 > |\theta_1|\}.$$ 

 We state the main result on the existence and uniqueness of the continuous-time Maximum Likelihood Estimator (MLE) using Notation \ref{notation:H-and-C}:

\begin{theorem}{($\boldTheta$-GrOU MLE with continuous-time observations)\\}
\label{theorem:mle_cts_vec}
Suppose that Assumptions \ref{assumption:A-known} and \ref{assumption:strong-solution-vector} hold. Assume that the $\boldTheta$-GrOU process $\YY$ is observed in continuous time and has finite second moments. Then, the MLE $\widehat{\boldsymbol{\theta}}_t$ on the compact set $\widehat{\boldsymbol{\Theta}}$ solves the equation $$\HH_t = [\HH]_t \cdot \widehat{\boldsymbol{\theta}}_t , \quad \text{for } t \geq 0.$$
Moreover, $\widehat{\boldsymbol{\theta}}_t$ satisfies the properties:
\begin{enumerate}
    \item We have $\det([\HH]_t) > 0 \ P_{t, \YY}-a.s.$ for $t \geq 0$ large enough and $$\widehat{\boldsymbol{\theta}}_t = [\HH]_t^{-1} \cdot \HH_t.$$
    \item The MLE $\widehat{\boldsymbol{\theta}}_t$ exists almost surely and uniquely under $\proba_{t,\YY}$.
\end{enumerate}
\end{theorem}
\begin{proof}
	See Appendix \ref{proof:theorem:mle_cts_vec}.
\end{proof}

This concludes the presentation of the two-parameter estimator and the MLE for the $\boldPsi$-GrOU process is defined next.

\subsection{The case of the $\psi$-GrOU process}
We focus on the $\boldPsi$-GrOU processes and define their likelihood along with the corresponding MLE. From \cite{basak2008asymptotic}, one deduces an intermediary result:
\begin{lemma}{(Adapted from Theorem 4.1, \cite{basak2008asymptotic})\\}
\label{lemma:square-kronecker-product}
Consider the $d\times d$ matrix ${\Kt}_t:=\int_0^t\YY_s \YY_s^\top ds$. This matrix is $\proba_{0}$-almost surely nonsingular for $t$ large enough in the sense that $\liminf_{t \rightarrow \infty} t^{-1}\lambda_{min}({\Kt}_t) > 0$ and we also have that $\lambda_{max}({\Kt}_t) = O(t) \ \proba_{0}-a.s.$ Here, $\lambda_{min}$ and $\lambda_{max}$ are respectively the smallest and largest eigenvalues (which are real since ${\Kt}_t$ is symmetric).
\end{lemma}
\begin{proof}
	See Appendix \ref{proof:lemma:square-kronecker-product}.
\end{proof}

We define the equivalent matrix to $\HH_t$ as follows
\begin{definition}
\label{definition:node-level-response}
	We define the node-level integrated response vector as $\Iit_t := - \int_0^t \YY_s \otimes \Si^{-1}d\YY_s^c$ such that $[\Iit]_t := \: {\Kt}_t \otimes \Si^{-1}$ for any $t \geq 0$.
\end{definition} 
As hinted by Equation \eqref{eq:Q-using-psi}, we first derive the likelihood under unrestricted network interactions before applying the network topology (as a linear transformation of the former). The corresponding $\boldPsi$-GrOU likelihood is formulated as follows:
\begin{proposition}
	\label{proposition:node-level-likelihood} 
	We consider the dynamics $d\YY_t = - \vectorise^{-1}(\boldPsi) \cdot \YY_{t-}dt + d\LL_t$ for some general parameter $\boldPsi \in \R^{d^2}$ such that Prop.\ \ref{proposition:node-level-psi-identiability} holds. The likelihood with respect to $\boldsymbol{\psi}$ is given by
$$\mathcal{L}_t(\boldsymbol{\psi}; \YY) = \exp\left(\boldsymbol{\psi}^\top \cdot  \Iit_t - \frac{1}{2}\boldsymbol{\psi}^\top \cdot  [\Iit]_t \cdot \boldsymbol{\psi}\right), \quad \boldPsi \in \R^{d^2}.$$
\end{proposition}
\begin{proof}
	See Appendix \ref{proof:proposition:node-level-likelihood}.
\end{proof}
\begin{remark}
	To obtain the MLE of $\vectorise(\rmQ(\boldPsi))$, we factor in the network topology $\An$ by transforming linearly $\boldPsi$ into $\vectorise(\Id + \An) \odot \boldPsi$ as given in Eq.\ \eqref{eq:Q-using-psi}. Therefore, there is no need to reformulate the likelihood given in Proposition \ref{proposition:node-level-likelihood}.
\end{remark}
Define a compact set $\widehat{\boldsymbol{\Psi}} \subseteq \R^{d^2}$ such that
\begin{equation}
	\label{eq:compact-psi-set}
	\widehat{\boldsymbol{\Psi}} \subseteq \Big\{\boldPsi \in \R^{d^2}: \psi_{d(i-1)+i} > n_i \sum_{j\neq i}|\psi_{d(j-1)+i}| > 0,\ \forall i \in \{1,\dots,d\}\Big\},
\end{equation}
that is where $\Q(\boldPsi)$ is diagonally dominant: this ensures the well-definedness of the GrOU process. Similarly to the $\boldTheta$-GrOU case, we denote by $\left(\proba_\YY^{\boldPsi}, \ \boldPsi \in \widehat{\boldsymbol{\Psi}}\right)$ the statistical models of \levy-driven Ornstein-Uhlenbeck processes with respect to the likelihood from Proposition \ref{proposition:node-level-likelihood} indexed on $\widehat{\boldsymbol{\Psi}}$.

We formulate an equivalent to Theorem \ref{theorem:mle_cts_vec} for the $\boldPsi$-GrOU process:

\begin{theorem}{($\boldPsi$-GrOU MLE with continuous-time observations)\\}
\label{theorem:mle_cts_vec_psi}
Suppose that Assumptions \ref{assumption:A-known} and \ref{assumption:strong-solution-vector} hold. Assume that the $\boldPsi$-GrOU process $\YY$ is observed in continuous time and has finite second moments. Then, the MLE $\widehat{\boldPsi}_t$ on the compact set $\widehat{\boldsymbol{\Psi}}$ solves the equation $$\Iit_t = [\Iit]_t \cdot \widehat{\boldPsi}_t , \quad \text{for } t \geq 0.$$
Moreover, $\widehat{\boldPsi}_t$ satisfies the properties:
\begin{enumerate}
    \item Since $\det([\Iit]_t) > 0 \ P_{t, \YY}-a.s.$ for $t \geq 0$ large enough and
   $$\widehat{\boldPsi}_t = [\Iit]_t^{-1} \cdot \Iit_t.$$
    \item The MLE $\widehat{\boldPsi}_t$ exists almost surely and uniquely under $\proba_{t,\YY}$.
\end{enumerate}
\end{theorem}
\begin{proof}
	Recall that $\int_0^t\YY_s \YY_s^\top ds$ is $\proba_{t,\YY}-a.s.$ invertible (see Lemma \ref{lemma:square-kronecker-product}). The same argument as for Theorem \ref{theorem:mle_cts_vec} can be applied.
\end{proof}

\section{Asymptotic theory for the MLEs}
\label{section:asymptotic-theory-continuous-times}
\myRed{Asymptotic properties of MLEs are a necessary step into formulating hypothesis tests necessary for sound inference.} In this section, we consider having access to a continuous flow of data and we derive the asymptotic normality of the afore-mentioned estimators as well as an augmented estimator on the whole $\Q$ matrix. We consider the estimators efficiency in the sense of H{\'a}jek-Le Cam's convolution theorem \citep[Section 2]{hajek1970efficiencyLan} 
under local asymptotic normality \citep[Chapter 5, Section 6]{LeCam1990}.

\subsection{Asymptotics for $\theta$-GrOU}
\label{section:network-level-asymptotics}
 
We now prove the consistency of the MLE for the $\boldTheta$-GrOU process on a compact set $\widehat{\boldsymbol{\Theta}}$.
\begin{proposition}{(Consistency of the estimator)\\}
\label{proposition:vec-cts-clt}
Suppose Assumptions \ref{assumption:A-known} \& \ref{assumption:strong-solution-vector} hold. Suppose that $(\YY_t,\ t \geq 0)$ satisfies Equation \eqref{eq:sde} and has finite second moments for $\boldsymbol{\theta} \in \boldsymbol{\widetilde{\Theta}}$ where $\boldsymbol{\widetilde{\Theta}}$ is a compact set in $\boldsymbol{{\Theta}}$. The MLE $\widehat{\boldsymbol{\theta}}_t$ is a consistent estimator under $\proba_{\YY}$ in the sense that $\widehat{\boldsymbol{\theta}}_t \xrightarrow{\ p \ } \boldsymbol{\theta}$ as $t \rightarrow \infty$. 
\end{proposition}

\begin{proof}
	See Appendix \ref{proof:proposition:vec-cts-clt}.
\end{proof}

We denote by $\left(\proba_\YY^{\boldTheta}, \ \boldTheta \in \widehat{\boldsymbol{\Theta}}\right)$ the statistical models of \levy-driven Ornstein-Uhlenbeck processes with respect to the likelihood from Proposition \ref{proposition:radon-nikodym} indexed on $\widehat{\boldsymbol{\Theta}}$. We obtain the local asymptotic normality for this sequence:
\begin{lemma}
\label{lemma:network-asymptotic-normal}
	 The family of statistical models $\left(\proba_\YY^{\boldTheta}, \ \boldTheta \in \widehat{\boldsymbol{\Theta}}\right)$ with respect to $\left(\mathcal{L}_t(\boldTheta;\YY): \boldTheta \in \widehat{\boldsymbol{\Theta}}\right)$ is locally asymptotically normal.
\end{lemma} 
\begin{proof}
	See Appendix \ref{proof:lemma:network-asymptotic-normal}
\end{proof}
The central limit theorem for $\widehat{\boldTheta}$ is given by:
\begin{theorem}
\label{th:clt-theta}
	Suppose that $(\YY_t,\ t \geq 0)$ satisfies Equation \eqref{eq:sde} and has finite second moments for $\boldsymbol{\theta} \in \boldsymbol{\widetilde{\Theta}}$. In addition, suppose that $\E\left[\exp(\boldTheta^\top \HH_t)\right]<\infty$ for $t$ large enough. Then, the MLE $\widehat{\boldTheta}_t$ satisfies under $\proba_{\YY}$
\begin{equation}
\label{eq:clt-stddev-time-dep}
   [\HH]_t^{1/2}\cdot \left(\widehat{\boldsymbol{\theta}}_t - \boldsymbol{\theta}\right) \xrightarrow{\ \mathcal{D} \ } \mathcal{N}(\boldsymbol{0}, I_{2 \times 2}), \quad \text{as $t \rightarrow \infty$.}
\end{equation}
Moreover, $\widehat{\boldTheta}_t$ is efficient in the sense of H{\'a}jek-Le Cam's convolution theorem.
\end{theorem}
\begin{proof}
	See Appendix \ref{proof:th:clt-theta}.
\end{proof}

\begin{corollary}
\label{corollary:2d-cts-clt}
Under the same assumptions as in Theorem \ref{th:clt-theta}, one obtains
$$t^{1/2}(\widehat{\boldsymbol{\theta}}_t - \boldsymbol{\theta}) \xrightarrow{\ \mathcal{D} \ } \mathcal{N}\left(\boldsymbol{0_2}, \boldsymbol{G}_\infty^{-1}\right), \quad \text{as $t \rightarrow \infty$,}$$
where $$\boldsymbol{G}_\infty :=
\begin{pmatrix}
\E\left( \langle \An \YY_\infty, \An \YY_\infty \rangle_{\Si} \right) & \E\left(  \langle \An \YY_\infty,  \YY_\infty \rangle_{\Si} \right)\\
\E\left( \langle \An \YY_\infty,  \YY_\infty \rangle_{\Si} \right) & \E\left( \langle  \YY_\infty,  \YY_\infty \rangle_{\Si} \right)
\end{pmatrix},$$
which is positive definite by Cauchy-Schwarz's inequality.
\end{corollary}
\begin{proof} Recall the result of Lemma \ref{lemma:HH_t-limit}; the continuous mapping theorem yields that $([\HH]_t/t)^{1/2}\rightarrow \boldsymbol{G}_\infty^{1/2}$ $\proba_\YY$-a.s.
By application of Theorem \ref{th:clt-theta} and Slutsky's lemma, the result follows directly.
\end{proof}

\subsection{Asymptotics for $\psi$-GrOU}
We present a decomposition of the MLE $\widehat{\boldsymbol{\boldPsi}}$ in the following lemma:
\begin{lemma}
	\label{lemma:martingale-decomposition}
	In the context of Theorem \ref{theorem:mle_cts_vec_psi}, we have that $$\widehat{\boldsymbol{\psi}}_t - \boldsymbol{\psi} = [\Iit]_t^{-1}\Mt_t, \quad t \geq 0,$$ 
where $\Mt_t := \int_0^t \YY_s \otimes d\WW_s$ is the (martingale) remainder vector and $[\Iit]_t = {\Kt}_t \otimes \Id$.
\end{lemma}
\begin{proof}
	See Appendix \ref{proof:lemma:martingale-decomposition}.
\end{proof}

Similarly to Section \ref{section:network-level-asymptotics}, we denote by $\left(\proba_\YY^{\boldPsi}: \ \boldPsi \in \widehat{\boldsymbol{\Psi}}\right)$ the statistical models of \levy-driven Ornstein-Uhlenbeck processes with respect to the likelihood from Proposition \ref{proposition:node-level-likelihood} indexed on $\widehat{\boldsymbol{\Psi}}$. We obtain the local asymptotic normality for this sequence:
\begin{lemma}
\label{lemma:node-asymptotic-normal}
	 The family of statistical models $\left(\proba_\YY^{\boldPsi}:\ \boldPsi \in \widehat{\boldsymbol{\Psi}}\right)$ with respect to $\left(\mathcal{L}_t(\boldPsi;\YY): \ \boldTheta \in \widehat{\boldsymbol{\Psi}}\right)$ is locally asymptotically normal.
\end{lemma} 
\begin{proof}
We can apply a similar argument to the proof of Lemma \ref{lemma:network-asymptotic-normal} in $\R^d$.
\end{proof}

As presented in Theorem \ref{th:clt-theta}, we obtain a central limit theorem for $\boldPsi$-GrOU with continuous-time observations as follows:

\begin{theorem}
\label{th:cv-cts-time}
Assume Assumptions \ref{assumption:A-known} \& \ref{assumption:strong-solution-vector} hold and that Prop.\ \ref{proposition:node-level-psi-identiability} can be applied. In addition, suppose that $\E\left[\exp(\boldPsi^\top \A^\Si_t) \right] < \infty$ for $t$ large enough and $\boldPsi \in \widehat{\boldsymbol{\Psi}}$. Then, $\widehat{\boldPsi}_t$ is consistent and we obtain:
$$t^{1/2}\left(\widehat{\boldsymbol{\psi}}_t - \boldsymbol{\psi}\right) \xrightarrow{\ \mathcal{D} \ } \mathcal{N}\left(\boldsymbol{0}_{d^2}, \E\left( \YY_\infty \YY_\infty^\top\right)^{-1} \otimes \Si \right), \quad \text{as $t \rightarrow \infty$,}$$
where we recall that $\E\left( \YY_\infty \YY_\infty^\top\right) = \int_0^\infty e^{-s\Q}  \Si   e^{-s\Q^\top} ds.$ Moreover, $\widehat{\boldPsi}_t$ is efficient in the sense of H{\'a}jek-Le Cam's convolution theorem.
\end{theorem}
\begin{proof}
	See Appendix \ref{proof:th:cv-cts-time}.
\end{proof}

We have derived an essential result for a general dynamics matrix (e.g.\ $\vectorise^{-1}(\boldPsi)$) with diagonal elements dominating the average off-diagonal parameters row-wise. We extend the result to include the network topology and derive a corollary for such a graph-constrained estimator as follows:
\begin{corollary}
	In the same setting as in Theorem \ref{th:cv-cts-time}, we have:
$$t^{1/2}\left\{\vectorise\left[\Q(\widetilde{\boldPsi}_t)\right] - \vectorise\Big[\Q(\boldPsi)\Big]\right\} \xrightarrow{\ \mathcal{D} \ } \mathcal{N}\left(\boldsymbol{0}_{d^2},\rmD_{\A} \cdot \E\left( \YY_\infty \YY_\infty^\top\right)^{-1} \otimes \Si \cdot \rmD_{\A}\right),\quad \text{as $t \rightarrow \infty$,}$$
where $\rmD_{\A} := \diag\left(\vectorise(\Id +\An)\right)$.
\end{corollary}
\begin{proof}
	By a property of the Hadamard product, observe that $\vectorise\left(\rmQ(\boldPsi)\right) =  \vectorise(\Id +\An) \odot \boldPsi =  \diag(\vectorise(\Id +\An)) \cdot \boldPsi = \rmD_{\A} \cdot \boldPsi$. By Theorem \ref{th:cv-cts-time}, the result follows directly.
\end{proof}
The term $\rmD_{\A}$ highlights the application of the network topology and yields a generalised form of the two-dimensional central limit theorem given in Corollary \ref{corollary:2d-cts-clt}. 

\begin{remark}
\label{remark:second-var-identifiability}
\cite{fasen2013} proved a similar central limit theorem but for the regression on $e^{-\Q}$ itself which remains an alternative to the MLE approach, but the identifiability issues mentioned below Remark \ref{remark:var}, in Section \ref{section:stationary-solution} hinders the direct estimation of $\mQ$ from $e^{-\Q}$.
\end{remark}
%

Until this point, we have assumed the availability of the adjacency matrix. However, sparse stochastic processes have become increasingly influential to handle high-dimension problems \citep{Gaiffas2019SparseOUProcess, ma2021sparseAutoregressions, belomestny2019sparse}. Regularisation  through a penalty on the model parameters is an important component of this literature and we show in the next section that general \levy-driven OU processes can be consistently transformed into GrOU processes in this context.

\section{Asymptotic theory of the Adaptive Lasso regularisation}
\label{section:asymptotics-adaptive-lasso}
\myRed{A key limitation of the MLEs is that the adjacency matrix should be fully specified: this limits the applicability of the GrOU process to datasets where the graph topology is known (as in Assumption \ref{assumption:A-known}). Also, regularisation techniques are a powerful tool to create sparse graph-like structure for high-dimensional problems \citep{chen2020community, ma2021sparseAutoregressions}. To prove that such tools can be used in our setting, we propose an Adaptive Lasso scheme and show its consistency and asymptotic normality.
\subsection{Adaptive Lasso Regularisation}
Applying an $L^1$-penalty on the dynamics matrix $\Q$ to the log-likelihood, called a Lasso regression, is a common practice to introduce sparsity into $\Q$. Then, the regularised process can be interpreted as a proper GrOU process. In addition, a parameter allows the practitioners to tune how sparse the then-estimated adjacency matrix should be. 
\begin{notation}
	The support of a vector or a matrix $x$ is denoted $\supp(x)$ and is defined as the set of indices of non-null coordinates of $x$. Additionnally, given a set of indices $\gI$, we denote by $x_{|\gI}$, the restriction of $x$ to the indices in $\gI$ and $x_{|\gI\times\gI}$ to the indices in $\gI\times\gI$.
\end{notation}
Similarly, Adaptive Lasso (AL) regularisation schemes leverages a penalty that takes into account any $t^{1/2}$-consistent estimator---the MLE herein---to provide better theoretical guarantees such as \emph{consistency in variable selection} \citep[Section 2.6]{buhlmann2011statistics} or asymptotic normality. The former is defined as the support of the estimator converging to the support of the true parameter asymptotically.
\subsection{Definition}
An AL scheme \citep{Gaiffas2019SparseOUProcess} applied on a \levy-driven OU process $\YY$ with \emph{unknown} dynamics matrix $\rmQ_0$ yields a GrOU-like process with non-trivial adjacency matrix $\A$. It is defined by
\begin{equation}
	\label{eq:adaptive-lasso}
	\widehat{\rmQ}_{\AL,t} := \argmax_{\Q} \ell_t(\rmQ) - \lambda \|\rmQ \odot |\widehat{\rmQ}_t|^{-\gamma} \|_1,
\end{equation}
for fixed parameters $\lambda \geq 0$ and $\gamma > 0$. Also, $\odot$ denotes the Hadamard product and the denominator of the penalty $|\widehat{\rmQ}_t|^{-\gamma}$ is evaluated elementwise. The log-likelihood is given by
$$\ell_t(\rmQ) = -\int_0^t \langle \rmQ \YY_{s}, d\YY^c_s\rangle_{\Si} - \frac{1}{2}\int_0^t\langle \rmQ \YY_s, \rmQ \YY_s\rangle_{\Si} ds,$$
with the corresponding $d \times d$ MLE matrix
$$\widehat{\rmQ}_t := - {\Kt}_t^{-1} \cdot \int_0^t \YY_{s} \cdot (d\YY^c_s)^\top.$$ 
The MLE components are almost-surely non-zero and penalise more the entries that are expected to be zero. }

\myRed{ Conditional on the knowledge of $\rmQ_0$, we show two oracle properties: (a) the scheme is consistent in variable selection: i.e.\ the support of $\widehat{\rmQ}_{\AL,t}$ converges to the support of the true parameter $\rmQ_0$ as $t \rightarrow \infty$; (b) the estimator is asymptotically normal as $t \rightarrow \infty$ over the support of the true parameter. For instance, a Lasso regression with Gaussian noise is not consistent \citep{zou2006adaptive}.
\subsection{Asymptotic properties}
 The parameter $\lambda$ is implicitly a function of the time horizon $t$, i.e.\ $\lambda = \lambda(t)$. We present an equivalent to Th.\ 4, \cite{Gaiffas2019SparseOUProcess} for \levy-driven OU processes.
\begin{theorem}{(Adapted from Th.\ 4, \cite{Gaiffas2019SparseOUProcess})}
\label{theorem:adaptive-lasso}
	Suppose that Assumptions \ref{assumption:A-known} \& \ref{assumption:strong-solution-vector} hold for a \levy-driven OU process $\YY$ with a true but unknown dynamics matrix $\rmQ_0$. For a fixed $\gamma > 0$, assume that $\lambda = \lambda(t)$ verifies $\lambda(t) t^{1/2} \rightarrow 0$ and $\lambda(t) t^{(1+\gamma)/2}\rightarrow \infty$ as $t \rightarrow \infty$. Then, under the assumption that $\rmQ_0$ is known, we obtain:
	\begin{enumerate}
		\item Consistency of the variable selection: $\proba\left(\supp(\widehat{\rmQ}_{\AL,t}) = \supp(\rmQ_0) \right) \rightarrow 1$ as $t \rightarrow \infty$.
		\item Asymptotic normality: 
		 $$t^{1/2}\left(\vectorise(\widehat{\rmQ}_{\AL,t}) - \vectorise(\rmQ_0)\right)_{|\gQ_0} \xrightarrow{\ \mathcal{D} \ } \mathcal{N}\left(\boldsymbol{0}_{d^2}, \E\left( \YY_\infty \YY_\infty^\top\right)^{-1}_{|\gQ_0 \times \gQ_0} \otimes \Si_{|\gQ_0 \times \gQ_0}  \right), \quad \text{as $t \rightarrow \infty$,}$$
		 where $\gQ_0 := \supp(\rmQ_0)$.
	\end{enumerate}
\end{theorem}
\begin{proof}
	See Appendix \ref{proof:theorem:adaptive-lasso}.
\end{proof}
Note that the adjacency matrix can therefore be estimated as follows:
$$(\widehat{\A}_{\AL, t})_{ij} := \indicator_{\{x\neq 0\}}\left((\widehat{\rmQ}_{\AL, t})_{ij}\right),$$
and the $\boldTheta$-GrOU inference can be applied next as a simplification step for high-dimensional applications although the impact of model misspecification is left for future research. 
Finally, the penalty parameter $\lambda$ can be chosen to reach a given sparsity criterion (trial-and-error) or by cross-validation \citep[Section 4.1]{Gaiffas2019SparseOUProcess}.}

\section{An extension to a volatility-modulated GrOU process}
\label{section:stoch-vol}

\myRed{Stationary noise distributions are usually too simplistic to explain the intrinsic variability of the data. Volatility modulation adds a stochastic scaling factor \citep{cai2016estimating, belomestny2019sparse} which follows its own dynamics to better represent exogenous source of uncertainty \citep{pigorsch2009DefinitionSemiPositiveMultOU, yang2020method} whilst a jump component helps to model unforeseen perturbations or rare calendar events \citep{barndorffVeraart2012StochVol}.}

We extend the framework of Section \ref{section:network-ou} to include a stochastic volatility modulation through a positive semidefinite Ornstein-Uhlenbeck (PSOU hereafter) process \citep{pigorsch2009multivariate, pigorsch2009DefinitionSemiPositiveMultOU} and a time-changed jump term. 

For the latter term, we adapt the univariate framework introduced in \cite{barndorffVeraart2012StochVol} to the multivariate case. We find that the volatility modulation and the jump term preserve the core properties of the model---i.e.\ its stationarity and ergodicity---which in turn imply that extensions of the results from Sections \ref{section:likelihood-and-estimators} and \ref{section:asymptotic-theory-continuous-times} hold.

\begin{notation}
	For a process $(\XX_t, \ t \geq 0) \subseteq \R^d$, we denote by $\varphi_{\XX_t}(\boldsymbol{u}):=\E\left[\exp\left(i\boldsymbol{u}^\top\XX_t\right)\right]$ its characteristic function at time $t$. Similarly, for an $\MdR$-valued process $(\boldsymbol{X}_t)$, we write $\varphi_{\boldsymbol{X}_t}(\boldsymbol{u}) := \E\left\{ \exp\left[i \trace(\boldsymbol{u}^\top\boldsymbol{X}_t)\right]\right\}$. Finally, we denote by $\log \varphi(\cdot)$ the  distinguished logarithm of $\varphi$ for an infinitely divisible distribution \citep[Lemma 7.6]{sato1999levy}
\end{notation}

\begin{remark}
	For a two-sided \Levy process $(\LLL_t, \ t \in \R) \subseteq \MdR$ and for adapted processes $(\boldsymbol{A}_t = (A_{ij,t}), \ t \geq 0)$, $(\boldsymbol{B}_t = (B_{ij,t}), \ t \geq 0)\subseteq \MdR$ with respect to $\LLL$, we denote by $\int_0^t \boldsymbol{A}_a d\boldsymbol{L}_s \boldsymbol{B}_s$ the matrix whose $(i,j)$-th element is given by $\sum_{k,l}\int_0^t A_{ik,s}  B_{lj,s} dL_{kl,s}$.
\end{remark}

\subsection{Model extension}
\label{section:stochastic-volatility-presentation}
Consider the continuous-time process $(\YY^{(v)}_t,\ t \geq 0)$ satisfying the stochastic differential equation
\begin{equation}
\label{eq:sde-ou-with-spou-jump-vol}
	d\YY^{(v)}_t = -\rmQ \YY^{(v)}_t dt + \Si_t^{1/2} d\WW_t +  d\JJ_{T_t}, \quad t \geq 0,
\end{equation}
where $(\Si_t, \ t \in \R)$ is a c{\`a}dl{\`a}g stochastic volatility (SV) process. In addition, $(\WW_t, \ t \in \R)$ is a d-dimensional Brownian motion process, $(T_t, \ t \in \R)$ is an increasing continuous process where $T_t \rightarrow \pm\infty$ $\proba_0$-a.s.\ as $t \rightarrow \pm\infty$ and $(\JJ_t, \ t \in \R)$ is a two-sided pure-jump \Levy process with characteristic triplet $(\gamma_\JJ, \boldsymbol{0}, \nu_{\JJ})$ with respect to the truncation function $\tau(\boldsymbol{z}) := \mathbb{I}_{\{\boldsymbol{x} \in \R^d : \|\boldsymbol{x}\| \leq 1\}}(\boldsymbol{z})$ (see Section \ref{section:levy-noise-details}).
Under standard regularity conditions (see Sections \ref{section:stoch-vol-term} \& \ref{section:pure-jump-component}), we know that the unique candidate for a stationary solution to Equation \eqref{eq:sde-ou-with-spou-jump-vol} is
\begin{equation}
\label{eq:stationary-solution-with-stoch-vol}
	\YY^{(v)}_t = \int_{-\infty}^t e^{-(t-s)\Q}\Si^{1/2}_s d\WW_s + \int_{-\infty}^t e^{-(t-s)\Q}d\JJ_{T_s}, \quad t \in \R,
\end{equation}
where both terms are well-defined by Corollary 4.1, \cite{basse2014stochastic} (see Sections \ref{section:stoch-vol-characteristic-function} \& \ref{section:time-changed-pure-jump-process-characteristic-function}). Note that we have now extended the domain from $t\geq 0$ to $t \in \R$ \citep[Remark 1]{Brockwell2009}. We study each term separately in Sections \ref{section:stochastic-volatility-component} \& \ref{section:pure-jump-component}.
%

Regarding the volatility process, consider a positive definite matrix $\VVV \in \Ss^{++}$ and a two-sided $d\times d$ matrix \Levy subordinator $(\LLL_t, \ t \in \R)$ \citep{barndorff2008matrix} such that $(\Si_t,\  t \in \R)$ is a stationary positive semidefinite Ornstein-Uhlenbeck (PSOU) process \citep{pigorsch2009DefinitionSemiPositiveMultOU}, i.e.\ given by

\begin{equation}
	\label{eq:sigma-statio-integral}
	\Si_t = \int_{-\infty}^t e^{-(t-s)\VVV} d\LLL_s e^{-(t-s)\VVV^\top}, \quad t \in \R.
\end{equation}
We recall the existence conditions of this stationary process in Section \ref{section:stoch-vol-inv-distribution-and-osd}.
\begin{assumption}
	We assume the independence between $\WW$, $\LLL$, $\JJ$ and $T$.
\end{assumption}

In the following two subsections, we characterise both terms presented in the stationary solution in Eq. \eqref{eq:stationary-solution-with-stoch-vol}. We then prove that the resulting process $(\YY^{(v)}_t)$ is mixing hence ergodic which requires additional definitions presented in the next section.

To prove the ergodicity of the model presented in Section \ref{section:stochastic-volatility-presentation}, we augment our framework with another class of stochastic mixed moving average (MMA) processes which have well-studied asymptotic behaviour such as the mixing and ergodic properties (see Section \ref{section:levy-bases-and-mma-processes}).

\subsection{\Levy bases, MMA processes and the mixing property}
\label{section:levy-bases-and-mma-processes}
In this section, we recall the definitions of \Levy bases, characteristic quadruplet and \levy-driven MMA processes.
\begin{definition}{\citep[Definition 3.1]{fuchs2013mixing}}
	A $d$-dimensional \Levy basis on $S\times \R$ is an $\R^d$-valued random measure $\Lambda = \{\Lambda(B): B\in\mathcal{B}_b(S \times \R)\}$ satisfying:
	\begin{enumerate}[(a)]
		\item the distribution of $\Lambda(B)$ is infinitely divisible for all $B \in \mathcal{B}_b(S\times \R)$;
		\item for any $n\in \N$ and pairwise disjoint sets $B_1,\dots,B_n \in \mathcal{B}_b(S \times \R)$ the random variables $\Lambda(B_1),\dots,\Lambda(B_n)$ are independent and
		\item for any pairwise disjoint sets $(B_i \in \mathcal{B}_b(S \times \R),\ i \in \N)$ satisfying $\bigcup_{n \in \N}B_n \in \mathcal{B}_b(S \times \R)$ the series $\sum_{n=1}^\infty \Lambda(B_n)$ converges almost surely and it holds that $\Lambda(\bigcup_{n\in \N}) = \sum_{n\in \N} \Lambda(B_n)$ almost surely.
	\end{enumerate}
\end{definition}
\begin{remark}
Here, we take $S= \Ss^{++}$ and note that a definition of $\Ss^{+}$-valued \Levy bases on $S\times \R$ would be formulated similarly.	
\end{remark}

As in \cite{fuchs2013mixing, Barndorff2011SupOU}, we restrict ourselves  to time-homogeneous and factorisable \Levy bases, i.e.\ with characteristic function $$\E\left[\exp\left(i\boldsymbol{z}^\top\Lambda(B)\right)\right] = \exp\left(\log \varphi(\boldsymbol{z}) \times \Pi(B)\right),\quad \text{for any $B \in \mathcal{B}_b(\Ss^{++}\times \R)$,}$$ 
where $\Pi = \pi \otimes\lambda^{leb}$ is the product of a probability measure $\pi$ on $\Ss^{++}$ and the Lesbesgue measure on $\R$ and $\boldsymbol{z} \mapsto \varphi(\boldsymbol{z})$ is the characteristic function of an infinitely divisible distribution (Section \ref{section:levy-noise-details}) characterised, say, by a triplet $(\gamma, \Si, \nu)$. 

Let $\widetilde{\gamma}(A):=\gamma$ and $\widetilde{\Si}(A):=\Si$ be trivial maps from $\Ss^{++}$ to, respectively, $\R^d$ and $\Ss^{+}$, and let  $\widetilde{\nu}(dx, A) := \nu(dx)$ be an extension of $\nu$ to $\R^d \times \Ss^{++}$.
As per Section 3, \cite{fuchs2013mixing} and p.\ 162 \cite{BarndorffVeraart2018AmbitFrameworkQuadruple}, any such quadruplet $(\widetilde{\gamma}, \widetilde{\Si}, \widetilde{\nu}, \pi)$ characterises completely in law a \Levy basis $\Lambda$ in the sense of Definition 33, \cite{BarndorffVeraart2018AmbitFrameworkQuadruple} where $\pi$ is then called the intensity measure (as an extension of the control measure from \cite{RajputRosinki1989SpectralIDdistributions}). \myRed{Indeed, $S \mapsto \int_S \widetilde{\gamma}(A) \pi(dA) = \gamma \pi(S)$ and $S \mapsto \int_S \widetilde{\Si}(A) \pi(dA) = \Si \pi(S)$ are respectively signed and unsigned measures on $\left(\Ss^{++}, \mathcal{B}(\Ss^{++})\right)$; while $\int_S \nu(dx, A) \pi(dA) = \nu(dx) \pi(S)$ is a \Levy measure on $\R$ for a fixed $S\in\Ss^{++}$.}

For the existence of integrals with respect to a \Levy basis, see Th.\ 3.2, \cite{fuchs2013mixing} and Th.\ 2.7, \cite{RajputRosinki1989SpectralIDdistributions}.
We recall the definition of multivariate MMA processes as follows:
\begin{definition}{\citep[adapted from Definition 3.3]{fuchs2013mixing}}\label{definition:mma} Let $\Lambda$ be an $\R^d$-valued \Levy basis on $S \times \R$ and let $f: \S \times \R \rightarrow \MndR$ be a measurable function. If the process 
$$\int_S \int_{\R} f(A, t-s)\Lambda(dA,ds),$$
exists in the sense of Theorem 3.2, \cite{fuchs2013mixing}, for all $t \in \R$, it is called an $n$-dimensional mixed moving average process (MMA for short). The function $f$ is said to be its kernel function.	
\end{definition}

Finally, we also recall the definition of mixing processes:
\begin{definition}
	\label{definition:mixing-def} A process $(\YY_t, \ t \in \R)$ is mixing if and only if, for any $t \in \R$
		 $$\proba\left(\{\YY_t\in A\}\cap \{\YY_{t+h}\in B\} \right) \longrightarrow{} \proba(\YY_t\in A)\proba(\YY_{t+h}\in B), \quad \text{as $h \rightarrow \infty$,}$$ for any $A\in \mathcal{F}_{-\infty}^t = \sigma(\{\YY_s, \ s \leq t\})$, $B\in \mathcal{F}_{t+h}^\infty = \sigma(\{\YY_s, \ s \geq t+h\})$.
\end{definition}
It is straightforward to observe that this implies ergodicity. \cite{fuchs2013mixing} adapt the mixing conditions given in \cite{maruyama1970infinitely} and \cite{rosinski1997equivalence} to the multivariate context and prove that \levy-driven MMA processes are mixing (Theorem 3.5 therein).

\subsection{Stochastic volatility component}
\label{section:stochastic-volatility-component}
This PSOU process and several extensions have been developed and studied in the last decade \citep{pigorsch2009DefinitionSemiPositiveMultOU, pigorsch2009multivariate, Barndorff2011SupOU, fuchs2013mixing, barndorffVeraart2012StochVol}. The ability to model specific marginal distributions  whilst remaining tractable gives a flexible and powerful method to augment our original model \citep[Sections 4.2 and 5]{pigorsch2009DefinitionSemiPositiveMultOU}.

Denote by $\boldsymbol{\rho}:\Ss\rightarrow \Ss$ the linear operator $\boldsymbol{X} \mapsto \VVV \boldsymbol{X} + \boldsymbol{X} \VVV^\top$ such that $e^{t\boldsymbol{\rho}}(\Ss)=\Ss$  \citep[Section 3]{pigorsch2009DefinitionSemiPositiveMultOU}.

\subsubsection{Invariant distribution and operator self-decomposability}
\label{section:stoch-vol-inv-distribution-and-osd}

The literature focuses on the existence and uniqueness of the invariant distribution given in Equation \eqref{eq:sigma-statio-integral} \citep{Masuda2004, pigorsch2009DefinitionSemiPositiveMultOU}.
 Suppose that
\begin{equation}
\label{eq:levy-subordinator-log-moment}
	\int_{\Ss^{+}} (\log\|\boldsymbol{Z}\| \vee 0)\nu_{\LLL}(d\boldsymbol{Z}) < \infty,
\end{equation}
then, there exists a unique invariant distribution $F_{\Si}$ (according to  Prop.\ 2.2, \cite{Masuda2004}, and Th.\ 4.1 \& 4.2, \cite{sato1984operator}) which we take in its matrix-valued form. The distribution $F_{\Si}$ is operator self-decomposable with respect to the linear operator $\boldsymbol{\rho}$ \citep[Prop.\ 4.3]{pigorsch2009DefinitionSemiPositiveMultOU}. Hence, $F_\Si$ is absolutely continuous if the support of $\LLL$ is non-degenerate (i.e.\ $\nu_{\LLL} (a+S)<1$ for any $a\in\Ss$ and $S\subseteq\Ss$ such that $\dim(S) \leq \dim(\Ss)-1$, see \cite{Yamazato1983}). In that case, note that this stationary distribution is almost surely concentrated on  $\Ss^{++}$ with respect to the Lebesgue measure \citep[Th.\ 4.4]{pigorsch2009DefinitionSemiPositiveMultOU}. According to Section \ref{section:stationary-solution}, \cite{pigorsch2009DefinitionSemiPositiveMultOU}, one can write for $t\geq 0$ that
$$\Si_t = e^{-t\VVV}\Si_0e^{-t\VVV^\top} + \int_0^t e^{-(t-s)\VVV}d\LLL_s e^{-(t-s)\VVV^\top},$$
or, in vectorised form, that
$$\vectorise(\Si_t) = e^{-t (\VVV \otimes \Id + \Id \otimes \VVV)}\vectorise(\Si_0) + \int_0^t e^{-(t-s) (\VVV \otimes \Id + \Id \otimes \VVV)}d\vectorise(\LLL_s).$$
Note that if $\VVV \in \Ss^{++}$ then $\VVV \otimes \Id + \Id \otimes \VVV \in \Ss^{++}$. 
%
In particular, $(\Si_t, \ t \geq 0)$ satisfies
\begin{equation}
	\label{eq:sde-stochastic-vol}
	d\Si_t = - \left(\VVV\Si_{t-} + \Si_{t-}\VVV^\top\right)dt + d\LLL_{t}, \quad t \geq 0,
\end{equation}
where $\Si_0 \in \Ss^+$. We conclude that $(\Si_t)$ is an MMA process \citep[Def.\ 3.3]{fuchs2013mixing}.
\begin{proposition}
\label{proposition:sigma-mma}
Suppose the framework given in Sections \ref{section:stochastic-volatility-presentation} \& \ref{section:stochastic-volatility-component} holds. Then, $(\Si_t,\ t\in \R)$ is an MMA process as given in Definition \ref{definition:mma}.
\end{proposition}
\begin{proof}
	See Appendix \ref{proof:proposition:sigma-mma}.
\end{proof}

\subsubsection{Characteristic function}
\label{section:stoch-vol-characteristic-function}
Recall that $(\LLL_t, \ t \in \R)$ is taken to be a two-sided matrix \Levy subordinator process: a process that is $\Ss^+$-increasing (such that $\LLL_t - \LLL_s \in \Ss^{+}$ for any $t > s$) and of finite variation \citep{barndorff2007positive}. 
It is characterised by a triplet $(\boldsymbol{\gamma}_{\LLL}, \boldsymbol{0}, \nu_{\LLL})$ where $\boldsymbol{\gamma}_{\LLL} \in \Ss^{+}$ and $\nu_{\LLL}$ is a \Levy measure on the space of positive semidefinite matrices $\Ss^{+}$ such that 
 \begin{equation}
 \label{eq:levy-matrix-subordinator-moment-conditions}
 	\int_{\Ss^{+}} (\|\ZZZ\| \wedge 1)\nu_{\LLL}(d\ZZZ) < \infty, \quad \text{and} \quad  \nu_{\LLL}(\{\boldsymbol{0}\}) = 0.
 \end{equation}
 According to Part 1, \cite{barndorff2015change}, given Equation  \eqref{eq:levy-matrix-subordinator-moment-conditions}, its characteristic function at time $t \in \R$ given by
\begin{equation}
\label{eq:characteristic-function-matrix-subordinator}
\varphi_{\LLL_t}(\UUU) := \exp\left\{ t\left[i\trace(\boldsymbol{\gamma}_{\LLL} \ZZZ) + \int_{\Ss^{+} \symbol{92}\{\boldsymbol{0}\}}\left(e^{i \trace(\ZZZ\UUU)}-1\right)\nu_{\LLL}(d\ZZZ)\right]\right\}, \quad \text{for} \  \UUU \in \Ss,	
\end{equation}
with respect to the truncation function $\widetilde{\tau}(\boldsymbol{X}) \equiv 0$ on $\MdR$ \cite[Part 1]{barndorff2015change}.

 Theorem 4.9, \cite{pigorsch2009DefinitionSemiPositiveMultOU} yields that if $\VVV\in \Ss^{++}$ and Equations \eqref{eq:levy-subordinator-log-moment} \& \eqref{eq:levy-matrix-subordinator-moment-conditions} hold, then
 the PSOU process $(\Si_t,\ t \in \R)$ is strictly stationary and its distribution is infinitely divisible with characteristic function
$$\varphi_\Si(\UUU) = \exp \left\{i \trace(\boldsymbol{\gamma}_\Si) + \int_{\Ss^{+} \symbol{92} \{\boldsymbol{0}\}}\left(e^{i \trace(\ZZZ \UUU)} - 1\right)\nu_\Si(d\ZZZ)\right\}, \quad \text{for} \ \UUU \in \Ss,$$
where $\boldsymbol{\gamma}_\Si := \boldsymbol{\rho}^{-1}(\boldsymbol{\gamma}_{\LLL}) \in \Ss^+$ and
	$$\nu_{\Si}(S) := \int_{0}^\infty \int_{\Ss^{+}\symbol{92}\{\boldsymbol{0}\}} \indicator_{S}\left(e^{-s\VVV}\ZZZ e^{-s\VVV^\top}\right)\nu_{\LLL}(d\ZZZ)ds, \quad \text{for} \ S \in \mathcal{B}(\Ss^+\symbol{92}\{\boldsymbol{0}\}).$$
In that case, note that $\nu_{\Si}(\Ss\symbol{92}\Ss^{+}) = 0$. 

Following the characterisation of $(\Si_t, \ t \in \R)$, we present the second part of the noise in Equation \eqref{eq:sde-ou-with-spou-jump-vol} which is a pure-jump time-changed \Levy process.

\subsubsection{Stochastic volatility of a multivariate OU process}
\label{section:stoch-vol-term}
Suppose that $(\Si_t)$ is strictly stationary. Let $a < b \in \R\cup\{\pm \infty\}$ and consider the process 
$$\mathbb{F}^{(1)}_{ab} = \int_{a}^b e^{-(b-s)\Q}\Si_s^{1/2}d\WW_s.$$
By Proposition \ref{proposition:sigma-mma}, $(\Si_t)$ is an \levy-driven MMA process  hence locally uniformly bounded as given by Theorem 4.3, (ii), \cite{Barndorff2011SupOU}. Therefore, for any $t \geq 0$, the integral
\begin{equation}
\label{eq:integrated-variance-lebesgue}
	\int_{-\infty}^t e^{-(t-s)\Q} \Si_s e^{-(t-s)\Q^\top}ds \ \text{is a Lebesgue integral of $(\Si_t)$ $\omega$-wise.}
\end{equation} 
We obtain the following distributional property for $F^{(1)}_{ab}$:
\begin{proposition}
	\label{proposition:f1-non-degenerate}
	 The distribution of $\mathbb{F}^{(1)}_{ab}$ is non-degenerate in the sense of \cite{Yamazato1983} for any $a<b \in \R\cup\{\pm\infty\}$
\end{proposition}
\begin{proof}
	See Appendix \ref{proof:proposition:f1-non-degenerate}.
\end{proof}

In the case when $a = -\infty$, we can prove the stationarity of the stochastic volatility term as follows
\begin{proposition}
\label{proposition:f1-statio}
	If $(\Si_t, \ t\in \R)$ is strictly stationary, then $(\mathbb{F}^{(1)}_{-\infty t}, \ t\in \R)$ is strictly stationary.
\end{proposition}
\begin{proof}
	See Appendix \ref{proposition:f1-statio:proof}.
\end{proof}

\subsection{Pure-jump component}
\label{section:pure-jump-component}
Let us next consider the pure-jump process $(\JJ_{T_t}, \ t \in \R)$. Suppose that 
\begin{equation}
\label{eq:log-moment-time-change-jump}
	\int_{\R^d}(\log\|\boldsymbol{z}\|\vee 0)\nu_\JJ(d\boldsymbol{z}) < \infty,
\end{equation}
as well as
\begin{equation}
\label{eq:square-time-change-jump}
	\int_{\R^d}(\|\boldsymbol{z}\|^2\wedge 1)\nu_\JJ(d\boldsymbol{z}) < \infty.
\end{equation}

\subsubsection{Characteristic function}
Since $\JJ$ is a pure-jump \Levy process and given Equation \eqref{eq:square-time-change-jump}, we have
$$\varphi_{\JJ_1}(\boldsymbol{u}) = \exp\left\{i\boldsymbol{u}^\top \gamma_\JJ + \int_{\R^d\symbol{92}\{\boldsymbol{0}\}} \left(e^{i\boldsymbol{u}^\top \boldsymbol{z}} - 1 - i\boldsymbol{u}^\top\boldsymbol{z}\tau(\boldsymbol{z})\right)d\nu_\JJ(d\boldsymbol{z})\right\}.$$
We recall that a stochastic process $X$ is adapted with respect to $T$ if $X$ is constant on any interval $[T_{t-},T_t]$ for any $t \in \R$. According to Lemma 10.14, \cite{Jacod1979ChangementDeTemps}, since $T_{-t} \rightarrow -\infty$ as $t \rightarrow -\infty$ and $T$ is continuous, then $\JJ$ is $T$-adapted. Similarly to Section 1.2.2, \cite{barndorffVeraart2012StochVol}, all the base properties of $\JJ$ carry over to the time-changed process. 
Therefore, the characteristic function of $(\JJ_{T_t})$ is given by
$$\varphi_{\JJ_{T_t}}(\boldsymbol{u}) = \exp\left\{T_t\cdot \log \varphi_{\JJ_1}(\boldsymbol{u}) \right\}.$$
This implies that $(\JJ_{T_t}, \ t \in \R)$ has a characteristic triplet $(T \gamma_\JJ, \boldsymbol{0},  T \otimes\nu_\JJ)$.

\subsubsection{Integrated time-changed pure-jump process}
\label{section:time-changed-pure-jump-process-characteristic-function}
Let $a < b \in \R\cup\{\pm \infty\}$ and consider the process
$$\mathbb{F}^{(2)}_{ab} := \int_a^b e^{-(b-s)\Q} d\JJ_{T_s}.$$ 
Consider the case where $a < b \in \R \cup \{\pm\infty\}$. According to Corollary 4.1, \cite{basse2014stochastic}, this integral in well-defined since $\|e^{-s\Q}\boldsymbol{x}\| \leq \|\boldsymbol{x}\|$ for any $s \geq 0$ and $\int_{\R^d} (\|\boldsymbol{z}\|^2 \wedge 1) \nu_\JJ(d\boldsymbol{z}) < \infty$ from Equation \eqref{eq:square-time-change-jump}. 

Conditional on the knowledge of $T$ and by independence between $T$ and $\JJ$, Lemma 15.1, p.\ 496, \cite{ContRama2004StochVol} yields
\begin{align*}
	\E\left[\exp\left(i \boldsymbol{u}^\top \mathbb{F}^{(2)}_{a t}\right)\Big| T \right]	&=\E\left[\exp\left( \int_a^t i \boldsymbol{u}^\top e^{-(t-s)\Q}d\JJ_{T_s}\right)\bigg|T \right]\\
	&= \exp\left[ \int_a^t  \log\varphi_{\JJ_1}\left(e^{-(t-s)\Q^\top}\boldsymbol{u}\right) dT_s\right].
\end{align*}
Additionally, $\mathbb{F}^{(2)}$ has a characteristic triplet with drift
$$\int_a^{T_t} e^{-(T_t-s)\Q}\gamma_\JJ ds + \int_a^{T_t}\int_{\R^d}e^{-(T_t-s)\Q}\boldsymbol{x}\left[\tau(e^{-(T_t-s)\Q}\boldsymbol{x}) - \tau(\boldsymbol{x})\right]\nu_\JJ(d\boldsymbol{x})ds,$$
and \Levy measure 
$$\int_{a}^{T_t}\int_{\R^d\symbol{92}\{\boldsymbol{0}\}}\indicator_{E}\left(e^{(T_t-s)\Q}\boldsymbol{x}\right)\nu_\JJ(d\boldsymbol{x})ds, \quad \text{$E \in \mathcal{B}(\R^d)$,}$$
by Lemma 3, \cite{kallsen2002time}.

\begin{proposition}
\label{proposition:f2-statio}
	$(\mathbb{F}^{(2)}_{-\infty t}, \ t \in \R)$ is strictly stationary.
\end{proposition}
\begin{proof}
The statement can be proved similarly to Proposition \ref{proposition:f1-statio} since $(T_t)$ is almost surely increasing which is not repeated here for the sake of brevity.
\end{proof}

\begin{proposition}
\label{proposition:si-and-j2-mixing} 
  Suppose the framework given in Sections \ref{section:stochastic-volatility-presentation}, \ref{section:stochastic-volatility-component} \& \ref{section:pure-jump-component} holds. Then, $(\mathbb{F}^{(2)}_{-\infty t}, \ t \in \R)$ is a \levy-driven MMA process hence mixing.	
 \end{proposition}
 \begin{proof}
 	See Appendix \ref{proposition:si-and-j2-mixing:proof}.
 \end{proof}
 
 Proposition \ref{proposition:si-and-j2-mixing} is important since the time-changed pure-jump component does not benefit from the Gaussian structure of $\mathbb{F}^{(1)}$ and MMA processes alleviate this complication.

\subsection{Stationarity and ergodicity}
 If Equations \eqref{eq:levy-subordinator-log-moment}, \eqref{eq:levy-matrix-subordinator-moment-conditions} and \eqref{eq:log-moment-time-change-jump} hold, then $(\YY^{(v)}_t, \ t \in \R)$ can be expressed as 
\begin{equation}
	\label{eq:sde-stoch-vol-solution-statio}
	\YY^{(v)}_t = \mathbb{F}^{(1)}_{-\infty t} + \mathbb{F}^{(2)}_{-\infty t}, \quad t \in \R,
\end{equation}
as given in Equation \eqref{eq:stationary-solution-with-stoch-vol}. Both terms have  characteristic functions given in Sections \ref{section:stoch-vol-characteristic-function} and \ref{section:time-changed-pure-jump-process-characteristic-function}. Similarly to the \levy-driven case in \cite{Masuda2004}, Equation \eqref{eq:sde-stoch-vol-solution-statio} yields solutions which have operator self-decomposable distributions given the independence between $\LLL$, $\WW$, $\JJ$ and $T$. Indeed, we write for $t_2 > t_1 \in \R$
$$\YY^{(v)}_{t_2} = e^{-(t_2-t_1)\Q}\YY^{(v)}_{t_1} + \mathbb{F}^{(1)}_{t_1 t_2}+ \mathbb{F}^{(2)}_{t_1 t_2},$$
where the three terms on the right-hand side are independent and the stationary distribution is absolutely continuous if the support of $\JJ$ is non-degenerate (see Section \ref{section:stoch-vol-inv-distribution-and-osd}).

\begin{proposition}
	Suppose the framework given in Sections \ref{section:stochastic-volatility-presentation}, \ref{section:stochastic-volatility-component} \& \ref{section:pure-jump-component}. If $(\Si_t, \ t \in \R)$ is strictly stationary, then $(\YY^{(v)}_t, \ t \in \R)$ is also strictly stationary.
\end{proposition}
\begin{proof}
%
From the stationarity of both right-hand side terms of Equation \eqref{eq:sde-stoch-vol-solution-statio} given by Propositions \ref{proposition:f1-statio} and \ref{proposition:f2-statio}, the result  follows directly.
\end{proof}


 We prove that $(\Si_t,\ t \in \R)$ and $(\YY^{(v)}_t, \ t \in \R)$ are ergodic in the following proposition:
 \begin{proposition}
 \label{proposition:yt-mixing-stoch-vol}
 	Suppose the framework given in Sections \ref{section:stochastic-volatility-presentation}, \ref{section:stochastic-volatility-component} \& \ref{section:pure-jump-component} holds. Then, $(\Si_t,\ t\in \R)$ and  $(\YY^{(v)}_t, \ t \in \R)$ are mixing and hence ergodic.
 \end{proposition}
\begin{proof}
	See Appendix \ref{proposition:yt-mixing-stoch-vol:proof}.
\end{proof} 
The mixing and ergodicity properties of $(\Si_t)$ are important for statistical inference on the stochastic volatility which is outside the scope of this article. We have proved that the resulting process is well-defined, stationary and ergodic. Therefore, conditional on the stochastic volatility, the estimator central limit theorems from Section \ref{section:asymptotic-theory-continuous-times} hold under stochastic volatility modulation.


\section{Conclusion}
In this article, we tackle the problem of modelling sparse interactions between multiple time series. \myRed{For this purpose, we have defined the \emph{Graph Ornstein-Uhlenbeck} process---a \levy-driven Ornstein-Uhlenbeck process adapted for graph structures---of which we propose two different configurations. We first consider a network-wide  parametrisation where there is only one parameter to characterise momentum across all nodes and another unique parameter for the network effect.} The first estimator is robust against the curse of dimensionality whilst only providing a scarce feedback on network interactions. Then, we consider an augmented version of this estimator with node-dependent momentum parameter and a different network effect for each of a node's neighbours. We derive the well-definedness, existence, uniqueness and efficiency of those estimators and \myRed{we prove three novel central limit theorems (CLT) as the time horizon goes to infinity for both MLEs and an Adaptive Lasso scheme. The CLTs are a necessary step for both hypothesis tests \citep{morales2000renyi} and quantifying uncertainty in the inference of graphical and/or high-dimensional time series. Finally, we extend the GrOU process to include both a stochastic volatility and jump terms which serves as an introduction towards flexible covariance structures for graphs.} We also show that the afore-mentioned properties and asymptotic behaviours hold under standard regularity and independence assumptions. 

\myRed{This work is the first step towards understanding the behaviour of continuous-time stochastic processes on graph structures. A limitation of the current formulation is the necessity to have a continuum of data available and recent theoretical studies and applications leverage high-frequency data sources to circumvent this issue \citep{kim2016asymptotic, brownlees2020estimation}. The asymptotic properties of the GrOU process in this context is left for future research. Next, the extensions to a time-dependent and stochastic modulation of the noise open up questions on the inference of sparse \levy-driven high-dimensional volatility or covariance structures \citep{tao2013optimal, belomestny2019sparse, cai2016estimating}. }

\appendix

\section*{Acknowledgments}
The authors gratefully acknowledge the financial support from the EPSRC Centre for Doctoral Training in Financial Computing and Analytics at University College London and Imperial College London (under the grant EP/L015129/1).

\section{Proofs}
\subsection{Notations}
To match the framework from \cite{kuchler1997exponentialbook}, we introduce the standard notation for Jacobian and Hessian matrix as follows:
\begin{notation}
\label{notation:jacobian-hessian-and-C_t}
Consider any mapping $(t, \boldTheta) \mapsto C_t(\boldTheta)$ where $t \in (0, \infty)$ and $\boldTheta \in \R^k$ for some $k\in\N$. Define $\boldsymbol{\Theta} := \{\boldTheta \in \R^k: \ |C_t(\boldTheta)| < \infty,\ \forall t \geq 0 \}$ such that for a fixed $t$, $\boldTheta \mapsto C_t(\boldTheta)$ is a twice-differentiable mapping on $\interior \ \boldsymbol{\Theta}$. The Jacobian   and the Hessian matrices with respect to $\boldTheta$ are denoted $\nabla{C_t}(\boldTheta)$ and $\laplace{C_t}(\boldTheta)$ and defined, respectively, as
$$\nabla{C_t}(\boldsymbol{\theta}) := 
\left(\frac{\partial C_t(\boldsymbol{\theta})}{\partial \theta_j},\ j \in \{1,\dots,k\}\right)^\top
 \ \text{and} \ \laplace{C_t}(\boldsymbol{\theta}):=
\left(\frac{\partial^2 C(\boldsymbol{\theta})}{\partial \theta_i \partial \theta_j},\ i,j \in \{1,\dots,k\}\right),
$$
for any $\boldsymbol{\theta} \in \interior \ \boldsymbol{\Theta}$.
\end{notation}

\subsection{Proof of Theorem \ref{theorem:mle_cts_vec}}
\label{proof:theorem:mle_cts_vec}

Consider the definition of a steep mapping from Section 2.2, \cite{kuchler1997exponentialbook}: with the notations from Notation \ref{notation:jacobian-hessian-and-C_t},
for a fixed $t \in (0,\infty)$, if $\boldsymbol{\theta} \mapsto C_t(\boldsymbol{\theta})$ is a differentiable convex map such that for any $\boldsymbol{\theta}_0 \in \interior \ \boldsymbol{\Theta}$ and  $\boldsymbol{\theta}_1 \in \boldsymbol{\Theta} \symbol{92} \interior \ \boldsymbol{\Theta}$ we have \begin{equation} 
        \frac{\partial}{\partial \alpha} C_t\left((1-\alpha)\boldsymbol{\theta_0} +\alpha \boldsymbol{\theta_1}\right) \longrightarrow \infty, \quad \text{as } \alpha \rightarrow  1, \label{eq:steep_limit}
    \end{equation}
    then $C_t$ is said to be steep. Note that a sufficient condition for this property to be true is to have $\boldsymbol{\Theta} = \R^k$ \citep[Section 2.2]{kuchler1997exponentialbook}. In this context, we define 
    $$C_t(\boldsymbol{\theta}) := \frac{1}{2} \int_0^t \langle \Q(\boldTheta) \YY_s;\Q(\boldTheta)\YY_s \rangle_\Si ds,$$
   and, from Notation \ref{notation:H-and-C}, recall that we have
    $$
\HH_t := - \begin{pmatrix}
\int_{0}^t\langle \An\YY_{s}, d\YY^c_s \rangle_\Si \\
\int_{0}^t\langle \YY_{s}, d\YY^c_s \rangle_\Si
\end{pmatrix} \ \text{such that} \ [\HH]_t = \begin{pmatrix}
		\int_0^t\langle \An \YY_{s},\An\YY_{s} \rangle_\Si ds & \int_0^t\langle \An\YY_{s},\YY_{s} \rangle_\Si ds\\
		 \int_0^t\langle \An\YY_{s},\YY_{s} \rangle_\Si ds & \int_0^t\langle \YY_{s},\YY_{s} \rangle_\Si ds
	\end{pmatrix}
,$$
and that $2C_t(\boldsymbol{\theta})= \boldsymbol{\theta}^\top \cdot 
	[\HH]_t \cdot \boldsymbol{\theta}$ by Equation \eqref{eq:laplace-theta-likelihood}.

\begin{proof}[Proof of Theorem \ref{theorem:mle_cts_vec}]

According to Proposition \ref{proposition:radon-nikodym}, the likelihood is defined for any $\boldTheta \in \R^2$ and \myRed{note that} $C_t(\boldTheta)$ is the cumulant generating function of $\HH_t$.
Theorem 8.2.1, \cite{kuchler1997exponentialbook} requires that (a) the said likelihood representation is minimal which is true since $\HH$ has almost surely affinely independent components \citep[Section 4]{kuchler1997exponentialbook}; (b) $\boldTheta \mapsto \boldTheta$ is injective which is trivially true; (c) the cumulant function $\boldTheta \mapsto C_t(\boldTheta)$ to be defined on a set $\boldsymbol{\Theta}$ independent of $t$ and to be steep - again, this is also true since by Lemma \ref{lemma:HH_t-limit} $\boldTheta\mapsto C_t(\boldTheta)$ is defined on $\R^2$ with $|[\HH]|_t < \infty$ a.s.\ componentwise for a fixed $t \geq 0$ large enough \citep[criterion in Section 2.2]{kuchler1997exponentialbook}.

In closed-form, we obtain from Proposition \ref{proposition:radon-nikodym} and Notation \ref{notation:jacobian-hessian-and-C_t} that:
\begin{align*}
    \nabla\ln\mathcal{L}_t(\boldsymbol{\theta};\YY) &=
\begin{pmatrix}
-\int_{0}^t\langle \An\YY_{s}, d\YY^c_s \rangle_{\Si} - \theta_1\int_0^t\langle \An\YY_{s},\An\YY_{s}\rangle_{\Si} ds - \theta_2 \int_0^t\langle \An\YY_{s}, \YY_{s}\rangle_{\Si} ds\\
-\int_{0}^t\langle \YY_{s}, d\YY^c_s \rangle_{\Si} - \theta_1\int_0^t\langle \An\YY_{s},\YY_{s}\rangle_{\Si} ds - \theta_2 \int_0^t\langle \YY_{s}, \YY_{s}\rangle_{\Si} ds 
\end{pmatrix},\\
&= \HH_t - [\HH]_t \cdot \boldsymbol{\theta}.
\end{align*}
Therefore, under necessary conditions, the MLE $\widehat{\boldsymbol{\theta}}_t$ up to time $t \geq 0$ large enough is given by $[\HH]_t^{-1} \cdot \HH_t$, i.e.
$$\widehat{\boldsymbol{\theta}}_t = \det([\HH]_t)^{-1}\begin{pmatrix}
\int_{0}^t\langle \YY_{u-}, d\YY^c_u\rangle_{\Si}  \int_0^t\langle \An\YY_{v}, \YY_{v}\rangle_{\Si} dv-\int_{0}^t\langle \An\YY_{u-}, d\YY^c_u\rangle_{\Si}  \int_0^t\langle \YY_{v}, \YY_{v}\rangle_{\Si} dv \\
\int_{0}^t\langle \An \YY_{u-}, d\YY^c_u\rangle_{\Si}  \int_0^t\langle \An\YY_{v}, \YY_{v}\rangle_{\Si} dv-\int_{0}^t\langle \YY_{u-}, d\YY^c_u\rangle_{\Si}  \int_0^t\langle \An\YY_{v}, \An \YY_{v}\rangle_{\Si} dv 
\end{pmatrix}.$$

Regarding the existence and uniqueness of the estimator, by application of Theorem 8.2.1, \cite{kuchler1997exponentialbook}, the maximum likelihood estimator exists and is unique if and only if $|\HH_t| < \infty \ P_{t,\YY}$-a.s.\ componentwise and if $\determinant([\HH]_t) > 0$. The first condition is satisfied again by Lemma \ref{lemma:HH_t-limit} whilst the second remains to be checked.

More precisely, the almost-sure positiveness of the determinant of the matrix $[\HH]_t$ is proved using Fubini's theorem along with the Cauchy-Schwarz's inequality (in both its inner product and integral forms) as follows:
\begin{align*}
    \int_0^t\langle \An\YY_{u},\An\YY_{u}\rangle_{\Si} du \times \int_0^t\langle \YY_{v}, \YY_{v}\rangle_{\Si} dv \geq \left(\int_0^t |\langle\An\YY_{u},\YY_{u}\rangle_{\Si} | du\right)^2 \quad \text{since $1 \geq \indicator\{u=v\}$}.
\end{align*}
 Here, (ineq)equalities are almost sure with respect to the martingale measure $P_{t, \YY}$. Finally, observe that the Cauchy-Schwarz inequalities are actually sharp since the diagonal of $\An$ is zero. More precisely, the $i$-th component of $\An \YY_u$ is not linearly dependent with the $i$-th component of $\YY_u$ since $\diag\left(\An\right) = \boldsymbol{0}$. Thus, we conclude since the determinant of $[\HH]_t$, denoted $\det([\HH]_t)$, has the following value:
\begin{align*}
    \det([\HH]_t) &= \int_0^t\langle \An\YY_{s},\An\YY_{s}\rangle_{\Si} ds \times \int_0^t\langle \YY_{s}, \YY_{s}\rangle_{\Si} ds - \left(\int_0^t\langle \An\YY_{s}, \YY_{s}\rangle_{\Si} ds\right)^2 > 0 \quad P_{t,\YY}-a.s.,
\end{align*}
for any $t \geq 0$. We have then proved that this estimator exists and is unique on $\R^2$. By convexity of the likelihood, it is also unique on any compact set on which it exists which concludes the proof.

\end{proof}

\subsection{Proof of Lemma \ref{lemma:square-kronecker-product}}
\label{proof:lemma:square-kronecker-product}
\begin{proof} To prove that $\liminf_{t \rightarrow \infty} t^{-1}\lambda_{min}({\Kt}_t) > 0$, one mainly uses Proposition \ref{proposition:ergodicity-vector} with the vectorised matrix $\YY_s \YY_s^{\top}$, denoted $\vectorise(\YY_s \YY_s^{\top})$ which yields 
$$\frac{1}{t} \int_{0}^{t} \vectorise(\YY_s \YY_s^{\top})ds \xrightarrow{\proba_{\boldsymbol{0}}-a.s.} \E_{t,\YY}\left(\vectorise(\YY_\infty \YY_\infty^{\top})\right) < \infty, \quad \text{as $t\rightarrow \infty$,}$$
which can clearly be written in its matrix form. Similarly to \cite{Masuda2004} and \cite{Gaiffas2019SparseOUProcess}, the stationary solution presented in Equation  \eqref{eq:stationary-solution-vec} gives away that the limiting (ergodic) quantity is positive definite a.s. since we have $$\E_{t,\YY}\left(\YY_\infty \YY_\infty^{\top}\right) = \int_0^\infty e^{-s\Q}\E\left(\LL_1 \LL_1^\top\right) e^{-s\Q^{\top}}ds,$$
which is positive definite. The result follows immediately and by extension, for $t$ large enough, we can deduce that ${\Kt}_t$ is almost-surely positive definite. 
\end{proof}

\subsection{Proof of Proposition \ref{proposition:node-level-likelihood}}
\label{proof:proposition:node-level-likelihood}

\begin{proof}[Proof of Proposition \ref{proposition:node-level-likelihood}]
Recall that Proposition \ref{proposition:node-level-psi-identiability} holds and we write $\widetilde{\Q} :=\vectorise^{-1}(\boldPsi)$. 
The logarithm of the likelihood defined Eq.\ \eqref{eq:radon-nikodym-derivative} evaluated at $\widetilde{\Q}$ is written as
$$-\int_0^t \langle \widetilde{\Q} \YY_{s}, d\YY^c_s\rangle_{\Si} - \frac{1}{2}\int_0^t\langle \widetilde{\Q} \YY_s, \widetilde{\Q}\YY_s\rangle_{\Si} ds.$$
The first term (up to its sign) is given by
$\int_0^t \langle\widetilde{\Q}\YY_s,d\YY^c_s\rangle_{\Si}$ which can be reformulated  as follows
\begin{align*}
	\int_0^t\langle\widetilde{\Q}\YY_s,d\YY^c_s\rangle_{\Si} &= \int_0^t\sum_{n=1}^d\sum_{m=1}^d\sum_{l=1}^d \widetilde{Q}_{ml}Y^{(l)}_s(\Si^{-1})_{mn}dY^{(n),c}_s\\
	&= \sum_{l=1}^d\sum_{m=1}^d\sum_{n=1}^d \widetilde{Q}_{ml}(\Si^{-1})_{mn}\int_0^t Y^{(l)}_s dY^{(n),c}_s\\
	&= \sum_{l=1}^d\sum_{m=1}^d \widetilde{Q}_{ml}\sum_{n=1}^d\left((\Si^{-1})_{mn}\int_0^t Y^{(l)}_s dY^{(n),c}_s\right)\\
	&= \boldsymbol{\psi}^\top \cdot \Id \otimes \Si^{-1}\cdot \int_0^t \YY_s \otimes d\YY_s^c,
\end{align*}
For the second term, one obtains that 
\begin{align*}
	\int_0^t\langle \widetilde{\Q} \YY_s, \widetilde{\Q}\YY_s\rangle_{\Si} &= \int_0^t\sum_{n=1}^d\sum_{m=1}^d\sum_{l=1}^d\sum_{j=1}^d \widetilde{Q}_{ml}(\Si^{-1})_{mn} Y^{(l)}_sY^{(j)}_s \widetilde{Q}_{nj}ds \\
	&= \sum_{n=1}^d\sum_{m=1}^d\sum_{l=1}^d\sum_{j=1}^d \widetilde{Q}_{ml} \int_0^tY^{(l)}_sY^{(j)}_sds (\Si^{-1})_{mn} \widetilde{Q}_{nj}\\
	&= \sum_{n=1}^d\sum_{m=1}^d\sum_{l=1}^d\sum_{j=1}^d \boldsymbol{\psi}_{d(l-1)+m} \int_0^tY^{(l)}_sY^{(j)}_sds (\Si^{-1})_{mn}  \boldsymbol{\psi}_{d(j-1)+n}.
	\end{align*}
	Since $\boldPsi = \vectorise(\widetilde{\Q})$, we have that
	\begin{equation}
		\label{eq:inner-product-as-second-order-structure}
		\int_0^t\langle\widetilde{\Q}\YY_s,\widetilde{\Q}\YY_s\rangle_{\Si}ds=\boldsymbol{\psi}^\top \cdot \int_0^t\YY_s\YY_s^\top ds \otimes \Si^{-1} \cdot \boldsymbol{\psi}.
	\end{equation}
	Finally, using that $(AC)\otimes(BD) = (A\otimes B)\cdot (C \otimes D)$ with $A,B,C,D$ four matrices with dimensions such that the products $AC$ and $BD$ are well-defined, one obtains
	$$\int_0^t\langle\widetilde{\Q}\YY_s,\widetilde{\Q}\YY_s\rangle_{\Si}ds=\boldsymbol{\psi}^\top \cdot\left(\Id\otimes\Si^{-1}\right)\cdot\left({\Kt}_t\otimes\Id\right) \cdot \boldsymbol{\psi} = \boldsymbol{\psi}^\top \cdot [\Iit]_t \cdot \boldsymbol{\psi},$$
	which concludes the proof. Given the stationary and square integrability of $\YY$, $\Iit$ and $[\Iit]$ are a.s.\ finite for $t < \infty$ and we have that $\boldsymbol{\Psi} := \left\{\boldPsi \in \R^{d^2}: |\boldsymbol{\psi}^\top \cdot  [\Iit]_t \cdot \boldsymbol{\psi}| < \infty, \ \forall t \geq 0 \right\} = \R^{d^2}$.
\end{proof}

\subsection{Proof of Proposition \ref{proposition:vec-cts-clt}}
\label{proof:proposition:vec-cts-clt}
\begin{proof}[Proof of Proposition \ref{proposition:vec-cts-clt}]
Similarly to Section \ref{proof:theorem:mle_cts_vec}, we define $$C_t(\boldsymbol{\theta}) := \frac{1}{2} \int_0^t \langle \Q(\boldTheta) \YY_s;\Q(\boldTheta)\YY_s \rangle_\Si ds.$$
We verify point-by-point the conditions of Theorem 8.3.1, \cite{kuchler1997exponentialbook} to show the consistency of the MLE under $P_{\YY}$. Let $\boldTheta \in \widehat{\boldsymbol{\Theta}}$. 

First, Lemma \ref{lemma:HH_t-limit}, we have
$$ t^{-1}[\HH]_t  \xrightarrow{\proba_{0}-a.s.}\boldsymbol{G}_\infty,$$
which is positive definite by Cauchy-Schwarz's inequality (see Section \ref{proof:theorem:mle_cts_vec}) and independent of $\boldTheta$ ensuring that
\begin{equation}
	\label{eq:condition-8.3.2-kuchler}
	 \|t^{-1} [\HH]_t - \boldsymbol{G}_\infty \| \xrightarrow{\ p \ } 0, \quad \text{as $t\rightarrow \infty$}.
\end{equation}
 
Then, we prove that $[\HH]_t^{-1/2}\cdot (\HH_t - \nabla C_t(\boldsymbol{\theta}))$ is stochastically bounded for $t \geq 0$ large enough. It requires to prove that for any $\epsilon > 0$, there exists $K > 0$ such that $$\sup_{\boldsymbol{\theta} \in \boldsymbol{\widetilde{\Theta}}}\proba_{t,\YY}\{|t^{-1/2}\boldsymbol{G}_\infty^{-1/2} \cdot (\HH_t - \nabla C_t(\boldsymbol{\theta}))| > K \} < \epsilon.$$
\myRed{Next, by ergodicity we have that:}
$$[\HH]_t^{-1/2} \cdot \HH_t = t^{1/2} ([\HH]_t/t)^{-1/2} \HH_t/t = O(t^{1/2}), \quad \proba_{0}-a.s.$$ as $t \rightarrow \infty$ independently of $\boldsymbol{\theta}$. Recall that $\nabla{C}_t(\boldsymbol{\theta}) = [\HH]_t\cdot \boldTheta$. We have that $[\HH]_t^{-1/2} \cdot \nabla{C}_t(\boldsymbol{\theta})$ has almost-surely finite components for any $t \geq 0$. Indeed, those are c{\`a}dl{\`a}g quantities which converge almost surely to a quantity that is $O(t^{1/2})$ as $t \rightarrow \infty$. Therefore, those components cannot explode with a probability larger than $0$ for any $t \geq 0$ large enough - say above some fixed $t_1 \geq 0$. This is valid on the compact set $\boldsymbol{\widetilde{\Theta}}$ hence we have proved the uniform stochastic boundedness of $\left\{[\HH]_t^{1/2} \cdot (\HH_t-\nabla{C}_t(\boldsymbol{\theta})): t \geq t_1 \right\}$ on $\widehat{\boldsymbol{\Theta}}$. Therefore, thanks to the latter and \eqref{eq:condition-8.3.2-kuchler},  Theorem 8.3.1, \cite{kuchler1997exponentialbook}
 yields that $\widehat{\boldsymbol{\theta}_t} \xrightarrow{\ p \ } \boldsymbol{\theta}$ as $t \rightarrow \infty$.
\end{proof}

\subsection{Proof of Lemma \ref{lemma:network-asymptotic-normal}}
\label{proof:lemma:network-asymptotic-normal}
We prove that the likelihood converges locally to a Gaussian shift experiment \citep{hajek1970efficiencyLan}. 
Since $d\YY^c_t = d\WW_t - \Q(\boldTheta)\YY_{t-}dt$, we decompose $\HH$ into a stochastic term $\HH^{(s)}$ and a drift term $\HH^{(d)}$ defined for any $t \geq 0$ by
$$\HH^{(s)}_t :=  - \begin{pmatrix}
\int_{0}^t\langle \An\YY_{s}, d\WW_s \rangle_\Si \\
\int_{0}^t\langle \YY_{s}, d\WW_s \rangle_\Si
\end{pmatrix} \ \text{and} \ \HH^{(d)}_t :=  \begin{pmatrix}
\int_{0}^t\langle \An\YY_{s}, \Q(\boldTheta)\YY \rangle_\Si ds \\
\int_{0}^t\langle \YY_{s}, \Q(\boldTheta)\YY_s \rangle_\Si ds
\end{pmatrix},$$
such that $\HH_t \overset{d}{=} \HH^{(s)}_t + \HH^{(d)}_t$. We note that 
\begin{equation}
\label{eq:drift-term-reformulation}
	 \HH^{(d)}_t \overset{d}{=} [\HH]_t \boldTheta,
\end{equation}
since  $\Q(\boldTheta) = \theta_2 \Id + \theta_1\An$.  According to Proposition \ref{proposition:radon-nikodym}, we define the log-likelihood $\ell_t(\boldTheta; \YY) := \ln \mathcal{L}_t(\boldTheta; \YY)$ for $\boldTheta \in \widehat{\boldsymbol{\Theta}}$. By the almost sure finiteness of $\HH$ and $[\HH]$, for any $\boldsymbol{h} \in \R^2$, we obtain as increment of the log-likelihood the following:
\begin{align*}
	\ell_t(\boldTheta + t^{-1/2}\boldsymbol{h}; \YY)  -\ell_t(\boldTheta; \YY) &= t^{-1/2}\boldsymbol{h}^\top \HH_t - t^{-1/2}\boldsymbol{h}^\top [\HH]_t \boldTheta - \frac{t^{-1}}{2} \boldsymbol{h}^\top [\HH]_t \boldsymbol{h} \\
	&= t^{-1/2}\boldsymbol{h}^\top \HH^{(s)}_t + t^{-1/2}\boldsymbol{h}^\top \HH^{(d)}_t - t^{-1/2}\boldsymbol{h}^\top [\HH]_t \boldTheta  - \frac{t^{-1}}{2} \boldsymbol{h}^\top [\HH]_t \boldsymbol{h}\\
	&= t^{-1/2}\boldsymbol{h}^\top \HH^{(s)}_t  - \frac{t^{-1}}{2} \boldsymbol{h}^\top [\HH]_t \boldsymbol{h}, \quad \text{by Equation \eqref{eq:drift-term-reformulation}.}
\end{align*} 
Then, we observe that: (i) $\HH^{(s)}$ is a square integrable martingale (under $\proba_{\YY}^{\boldTheta}$) by the square integrability of $\YY$; (ii) $\HH^{(s)}$ is continuous hence $\HH^{(s)}_t-\HH^{(s)}_{t-}=0$ because $\YY$ is c{\`a}dl{\`a}g; (iii) $t^{-1}[\HH]_t \rightarrow \boldsymbol{G}_{\infty}$, $\proba_0-$almost surely (by Theorem \ref{proposition:ergodicity-vector}) hence in probability where $\boldsymbol{G}_{\infty}$ is  a (deterministic) positive definite matrix. This also yields the convergence of the covariance matrix of $\HH^{(s)}_t$ as $t \rightarrow \infty$. Then, by Theorem A.7.7, \cite{kuchler1997exponentialbook}, we obtain the pairwise convergence
$$\left(t^{-1/2} \HH^{(s)}_t,\ \frac{t^{-1}}{2} [\HH]_t \right) \xrightarrow{\ \mathcal{D}\ } \left( \boldsymbol{G}_\infty^{1/2}\boldsymbol{\eta}, \ \frac{1}{2}\boldsymbol{G}_\infty \right), \quad \text{as $t \rightarrow \infty$,}$$
where $\boldsymbol{\eta}$ is a $2$-dimensional standard normal random variable and therefore
\begin{equation}
\label{eq:likelihood-increment}
	l(\boldTheta + t^{-1/2}\boldsymbol{h}; \YY)  -l(\boldTheta; \YY) \xrightarrow{\mathcal{D}} \boldsymbol{h}^\top\boldsymbol{G}_\infty^{1/2}\boldsymbol{\eta} - \frac{1}{2}\boldsymbol{h}^\top \boldsymbol{G}_\infty\boldsymbol{h}, \quad \text{as $t \rightarrow \infty$.}
\end{equation}
We conclude that the family of statistical models $\left(\proba_\YY^{\boldTheta}, \ \boldTheta \in \widehat{\boldsymbol{\Theta}}\right)$ is locally asymptotically normal since $\boldsymbol{G}_\infty$ is deterministic. 

\subsection{Proof of Theorem \ref{th:clt-theta}}
\label{proof:th:clt-theta}
\begin{proof}[Proof of Theorem \ref{th:clt-theta}]	
	Using the notations of Section \ref{proof:proposition:vec-cts-clt} (proof of Proposition \ref{proposition:vec-cts-clt}), we have that $t\mapsto \nabla C_t(\boldTheta) = [\HH]_t\cdot \boldTheta$ is continuous with bounded variation on compact intervals as a time integral of an $L^1$ integrand. 
	 Condition 8.3.2, \cite{kuchler1997exponentialbook} holds since: (a) $\HH_t$ is continuously-valued in time; (b) $t^{-1}\laplace C_t(\boldTheta) = t^{-1}[\HH]_t \rightarrow \boldsymbol{G}_{\infty}$ under $\proba_{\YY}$ which is positive definite; (c) by Fubini-Tonelli's theorem, we have the same limit for the expected information matrix
$$t^{-1} \cdot \E\left([\HH]_t\right) \xrightarrow{\proba_{0}-a.s.} \boldsymbol{G}_\infty,$$
since $\YY_s$ has finite second moments for any $s \geq 0$. 
Hence, by Theorem 8.3.4, \cite{kuchler1997exponentialbook}, we obtain that 
$$[\HH]_t^{1/2}\left(\widehat{\boldsymbol{\theta}}_t - \boldsymbol{\theta}\right) \xrightarrow{\ \mathcal{D} \ } \mathcal{N}(\boldsymbol{0}, I_{2 \times 2}), \quad \text{as $t \rightarrow \infty$.}$$

Regarding the asymptotic efficiency, recall that $\widehat{\boldTheta} := [\HH]_t^{-1}\HH_t$ and from Lemma \ref{lemma:network-asymptotic-normal} and its proof (Equations \eqref{eq:drift-term-reformulation} \& \eqref{eq:likelihood-increment}), we obtain that that  $t^{1/2}(\widehat{\boldTheta}-\boldTheta)$ is asymptotically Gaussian with an asymptotic variance of $\boldsymbol{G}_\infty$ which is the corresponding Fisher information matrix. According to the H{\'a}jek-Le Cam's convolution theorem for locally asymptotically normal experiments \citep{LeCam1990}, we obtain the H{\'a}jek-Le Cam asymptotic efficiency of the estimator.
\end{proof}

\subsection{Proof of Lemma \ref{lemma:martingale-decomposition}}
\label{proof:lemma:martingale-decomposition}
\begin{proof}[Proof of Lemma \ref{lemma:martingale-decomposition}]
 Similarly to Example 8.3.6, \cite{kuchler1997exponentialbook}, we notice that Equation \eqref{eq:sde} in particular for the continuous part of $\YY$ yields $d\YY^{c}_t = -\rmQ\YY_tdt + d\WW_t$. Since $\widehat{\boldPsi}$ can be written $[\Iit]_t^{-1}\Iit_t$, we have $\widehat{\boldsymbol{\psi}}_t = [\ \Iit\ ]_t^{-1} ([\Iit]_t \boldsymbol{\psi} + \Mt_t) = \boldsymbol{\psi} + [\Iit]_t^{-1}\Mt_t$ where $\Mt_t = \int_0^t\YY_t \otimes d\WW_t$. In addition, $\Mt_t$ is a martingale under $\proba_{0}$ since $\YY_t \in L^2(\R^d)$. 
\end{proof}

\subsection{Proof of Theorem \ref{th:cv-cts-time}}
\label{proof:th:cv-cts-time}
\begin{proof}[Proof of Theorem \ref{th:cv-cts-time}]
Denote $\boldsymbol{G}^\boldPsi_\infty := \E(\YY_\infty \YY^{\top}_\infty) \otimes \Si^{-1}$. The proof is similar to that of Prop.\ \ref{proposition:vec-cts-clt} and we define $\widetilde{\Q} := \vectorise^{-1}(\boldPsi)$ as well as
$$C_t(\boldPsi) := \frac{1}{2}\int_0^t \langle \widetilde{\Q} \YY_s, \widetilde{\Q} \YY_s\rangle_\Si ds = \frac{1}{2}\boldPsi^\top \cdot [\Iit]_t \cdot \boldPsi.$$ 
Observe that $\nabla^2 C_t(\boldPsi) = [\Iit]_t = {\Kt}_t \otimes \Si^{-1}$ and we have by ergodicity
$$ t^{-1}[\Iit]_t\longrightarrow \E(\YY_\infty \YY^{\top}_\infty) \otimes \Si^{-1}= \boldsymbol{G}^\boldPsi_\infty \quad \proba_0-a.s., \quad \text{as $t\rightarrow \infty$,}$$
since $t^{-1}\int_0^t Y^{(i)}_s Y^{(j)}_s ds \rightarrow  \E\left( Y^{(i)}_\infty Y^{(j)}_\infty \right) \ \proba_{0}-a.s.$ as $t \rightarrow \infty$. Note that $t^{-1} \E\left([\Iit]_t\right)$ converges to the same limit and that $$[\Iit]_t^{-1/2}\Iit_t = t^{1/2}\cdot ([\Iit]_t/t)^{-1/2} \cdot \Iit/t = O(t^{1/2})\quad \proba_0-a.s., \ \text{as $t \rightarrow \infty$,}$$
since $\boldsymbol{G}^\boldPsi_\infty$ has almost-surely finite components and positive definite.

Theorem 8.3.4, \cite{kuchler1997exponentialbook} gives a central limit theorem with $[\Iit]_t^{1/2}$ as the scaling matrix.
Since the stochastic boundedness and finite variation over bounded intervals of $\left\{[\Iit]_t^{-1}\left(\Iit_t - \nabla C_t(\boldPsi)\right) : t \geq 0 \right\}$ on $\widehat{\boldsymbol{\Psi}}$ can be proved similarly to the proof of Proposition \ref{proposition:vec-cts-clt}, we can now apply Theorem 8.3.4, \cite{kuchler1997exponentialbook} and conclude that, as $t \rightarrow \infty$, 
$$[\Iit]_t^{-1/2}\left(\widehat{\boldsymbol{\psi}}_t - \boldsymbol{\psi}\right) \xrightarrow{\ \mathcal{D} \ } \mathcal{N}\left(\boldsymbol{0}_{d^2}, \Isquare \right).$$ 

and, by Slutsky's lemma,
$$t^{1/2}\left(\widehat{\boldsymbol{\psi}}_t - \boldsymbol{\psi}\right) \xrightarrow{\ \mathcal{D} \ } \mathcal{N}\left(\boldsymbol{0}_{d^2}, \E(\YY_\infty  \YY^{\top}_\infty)^{-1} \otimes \Si\right).$$
\end{proof}

\subsection{Proof of Theorem \ref{theorem:adaptive-lasso}}
\label{proof:theorem:adaptive-lasso}

\begin{proof}\myRed{We first show the asymptotic normality property before leveraging it to prove the consistency in variable selection of the Adaptive Lasso.
Since $d\YY^c_t = -\Q_0 \YY_{t-}dt + d\WW_t$, we center the log-likelihood around $\rmQ_0$ as follows:
$$\ell_t(\rmQ) = -\int_0^t \langle \rmQ \YY_{s}, d\WW_s\rangle_{\Si} - \frac{1}{2}\int_0^t\langle (\rmQ-\rmQ_0) \YY_s, (\rmQ-\rmQ_0) \YY_s\rangle_{\Si} ds +  \frac{1}{2}\int_0^t\langle \rmQ_0 \YY_s, \rmQ_0 \YY_s\rangle_{\Si} ds.$$
such that $\ell_t(\rmQ)-\ell_t(\rmQ_0) =  -\int_0^t \langle (\rmQ-\rmQ_0) \YY_{s}, d\WW^c_s\rangle_{\Si} - \frac{1}{2}\int_0^t\langle (\rmQ-\rmQ_0) \YY_s, (\rmQ-\rmQ_0) \YY_s\rangle_{\Si} ds$. Without loss of generality, we change the penalty rate from $\lambda$  to $\lambda t$ since it does not depend on $\rmQ$. Then, by writing $\rmQ = \rmQ_0 + t^{-1/2}\rmM$ for some $\rmM \in \MdR$, we have 
\begin{align*}
	t^{1/2}\left(\widehat{\rmQ}_{\AL,t} - \rmQ_0\right)_{|\gQ_0} &= \argmax_{\rmM \in \MdR} \kappa^1_t(\rmM) + \kappa^2_t(\rmM),
\end{align*}
where
\begin{align*}
	\begin{cases}
		\kappa^1_t(\rmM) &= -t^{-1/2}\int_0^t \langle \rmM \YY_{s-}, d\WW_s\rangle_{\Si} - \frac{t^{-1}}{2}\int_0^t\langle \rmM \YY_s, \rmM \YY_s\rangle_{\Si} ds,\\
		\kappa^2_t(\rmM) &= \lambda t \|\rmQ_0 \odot |\widehat{\rmQ}_t|^{-\gamma} \|_1 - \lambda t \|(\rmQ_0 + t^{-1/2}\rmM) \odot |\widehat{\rmQ}_t|^{-\gamma} \|_1.\\
	\end{cases}
\end{align*}}

\myRed{\paragraph{For the first function $\kappa^1_t$}
We define the process $\ermM_t:= \int_0^t \langle \rmM \YY_{s-}, d\WW_s\rangle_{\Si}ds$ and note that $[\ermM]_t = \int_0^t\langle \rmM \YY_{s-}, \rmM \YY_{s-}\rangle_{\Si} ds$ such that $\kappa^1_t(\rmM) = - t^{-1/2} \ermM_t -  \frac{1}{2}t^{-1} [\ermM]_t.$}

\myRed{We proceed similarly to the proof of Lemma \ref{lemma:network-asymptotic-normal} (Appendix \ref{proof:lemma:network-asymptotic-normal}): (i) $\ermM_t$ is a square integrable martingale (under $\proba_{\YY}$) by the square integrability of $\YY$; (ii) $\ermM_t$ is continuous hence $\ermM_t-\ermM_{t-}=0$ because $\YY$ is c{\`a}dl{\`a}g; (iii) we have the following ergodic convergence (Prop.\ \ref{proposition:ergodicity-vector}) as $t \rightarrow \infty$:
$$t^{-1}\int_0^t\langle \rmM \YY_s, \rmM \YY_s\rangle_{\Si} ds \rightarrow \E\left(\langle \rmM \YY_\infty, \rmM \YY_\infty \rangle_\Si\right), \proba_{\boldsymbol{0}}-a.s.$$
and we obtain the convergence in probability, too. Similarly to the proof of Prop.\ \ref{proposition:node-level-likelihood} (App.\ \ref{proof:proposition:node-level-likelihood}), namely Eq.\ \eqref{eq:inner-product-as-second-order-structure}, we have
$$\E\left(\langle \rmM \YY_\infty, \rmM \YY_\infty \rangle_\Si\right) = \vectorise(\rmM)^\top \cdot \rmK_\infty \otimes \Si^{-1} \cdot \vectorise(\rmM) =: K^1_\infty,$$
where $\rmK_\infty := E\left(\YY_\infty \YY_\infty^\top\right)$. From those conditions, Th.\ A.7.7, \cite{kuchler1997exponentialbook} yields the pairwise convergence
\begin{equation}
	\label{eq:pairwise-adaptive-lasso}
	\left(t^{-1/2} \ermM_t,\ \frac{t^{-1}}{2} \langle\ermM\rangle_t \right) \xrightarrow{\ \mathcal{D}\ } \left((K^1_\infty)^{1/2}Z, \ \frac{K^1_\infty}{2}\right), \quad \text{as $t \rightarrow \infty$,}
\end{equation}
where $Z$ is a standard Gaussian random variable on the same probability space as $\ermM_t$. Hence, since $K^1_\infty$ is deterministic and positive, we obtain $ t^{-1/2}\ermM_t \xrightarrow{\ \mathcal{D}\ } (K^1_\infty)^{1/2}Z$ as $t \rightarrow \infty$. A distribution trick \citep[Section 6.6]{Gaiffas2019SparseOUProcess} uses a $d \times d$ Gaussian matrix $\mZ$  with mean zero and covariance $\Cov(\vectorise(\mZ),\vectorise(\mZ)) = \mK_\infty \otimes \Si^{-1}$, that is, $\Cov(\mZ_{ij}, \mZ_{kl}) = (\mK_\infty)_{ik} \cdot \Si^{-1}_{jl}$. Then, for any $\rmM \in \MdR$, $\trace(\rmM \mZ)$ is a Gaussian random variable with mean zero and variance
\begin{align*}
	\Var\left(\trace(\rmM \mZ)\right) &= \sum_{ijkl}\rmM_{ji} \rmM_{lk} \Cov(\mZ_{ij}, \mZ_{kl}) = \sum_{ijkl}\rmM_{ji} \cdot (\mK_\infty)_{ik} \cdot \Si^{-1}_{jl} \cdot \rmM_{lk} = K^1_\infty.
\end{align*}
Using this reformulation and \eqref{eq:pairwise-adaptive-lasso}, we have
$$\kappa^1_t(\rmM) \xrightarrow{\ \mathcal{D}\ } - K^1_\infty/2  - \trace\left(\rmM \mZ\right), \quad \text{as $t \rightarrow \infty$.}$$
\paragraph{For the second function $\kappa^2_t$} Explicitly, we have
$$\kappa^2_t(\rmM) = -\lambda t \sum_{ij}|\widehat{\rmQ}_{t,ij}|^{-\gamma} \left(|\rmQ_{0,ij} + t^{-1/2}\rmM_{ij}| - |\rmQ_{0,ij}|\right).$$
For $(i,j)\in \gQ_0$, we have $t^{1/2} |^{-\gamma} \left(|\rmQ_{0,ij} + t^{-1/2}\rmM_{ij}| - |\rmQ_{0,ij}|\right) \rightarrow \sign(\rmQ_{0,ij}) |\rmM_{ij}|$ as $t\rightarrow \infty$. By the consistency of the MLE and the continuous mapping theorem, we have $|\widehat{\rmQ}_{t,ij}|^{-\gamma} \xrightarrow{\ p \ } |\rmQ_{0,ij}|^{-\gamma} > 0$. By assumption $\lambda t^{1/2} = \lambda(t) t^{1/2} \rightarrow 0$ and therefore we have
 $$\lambda(t) t^{1/2}|\widehat{\rmQ}_{t,ij}|^{-\gamma} \left(|\rmQ_{0,ij} + t^{-1/2}\rmM_{ij}| - |\rmQ_{0,ij}|\right) \xrightarrow{\ p \ } 0, \quad \text{as $t \rightarrow \infty$.}$$
Otherwise, for $(i,j)\in \widebar{\gQ}_0 := \{(k,l)\not\in\gQ_0: 1\leq k,l \leq d\}$, we proceed similarly but, from $t^{1/2}$-consistency of the MLE, we have $|t^{1/2}\widehat{\rmQ}_{t,ij}| = o_p(1)$. Hence, since $\lambda(t) t^{(\gamma+1)/2} \rightarrow \infty$ and $\gamma > 0$, we obtain that
 $$- \lambda(t) t^{(\gamma+1)/2}|t^{1/2}\widehat{\rmQ}_{t,ij}|^{-\gamma}|\rmM_{ij}|  \xrightarrow{\ p \ } -\infty, \quad \text{if $\rmM_{ij} \neq 0$, as $t \rightarrow \infty$.}$$
 and this quantity is zero if $\rmM_{ij} = 0$. We summarise the resulting asymptotic properties as follows:
$$
\kappa^1_t(\rmM) +\kappa^2_t(\rmM)  \xrightarrow{\ \mathcal{D}\ } \begin{cases}
 	-\infty, & \text{if } \supp(\rmM) \cap \widebar{\gG}_0 \neq \emptyset,\\
 	- K^1_\infty/2 -  \trace\left(\rmM \mZ\right), & \text{otherwise}.
 \end{cases}
$$
Suppose that the objective function is well-defined, i.e.\ $\supp(\rmM) \cap \widebar{\gG}_0 = \emptyset$: the entries from $\rmM$ with indices outside the support of $\mQ_0$ are all zero. In that case, we have $\trace(\rmM\mZ) = \vectorise(\rmM_{|\gG_0})^\top \cdot \vectorise(\mZ^\top_{|\gG_0})$. Also, we have $K^1_\infty = \vectorise(\rmM_{|\gG_0})^\top \cdot \left(\rmK_\infty \otimes \Si^{-1} \right)_{|\gG_0 \times \gG_0} \cdot \vectorise(\rmM_{|\gG_0})$ and we conclude that the objective function is quadratic in $\vectorise(\rmM_{|\gG_0})$. Its maximum in the limit $t\rightarrow \infty$, denoted $\widehat{\rmM}$, is given by
$$\vectorise(\widehat{\rmM}_{|\gG_0}):= - \left(\rmK_\infty^{-1} \otimes \Si \right)_{|\gG_0 \times \gG_0} \cdot \vectorise(\mZ^\top_{|\gG_0}) , \quad \text{and} \quad \widehat{\rmM}_{|\widebar{\gG}_0} := 0.$$
Given the covariance structure of $\mZ$, $\vectorise(\widehat{\rmM}_{|\gG_0})$ is a centred Gaussian random vector with covariance $\left(\rmK_\infty^{-1} \otimes \Si\right)_{|\gG_0 \times \gG_0}$ which   the proof of asymptotic normality.}

\myRed{\paragraph{Consistency in variable selection} We first note that the asymptotic normality of $\widehat{\rmQ}_{\AL,t}$ on $\gG_0$  yields that $\proba\left((\widehat{\rmQ}_{\AL,t})_{ij} \neq 0 \right) \rightarrow 1$ if $(i,j) \in \gG_0$. It remains to show that $\proba\left((\widehat{\rmQ}_{\AL,t})_{ij} = 0 \right) \rightarrow 1$ if $(i,j) \in \widebar{\gG}_0$. Suppose that $(\widehat{\rmQ}_{\AL,t})_{ij} \neq 0$. We first take the derivative of the objective function with respect to the $(i,j)$-th parameter of $\widehat{\rmQ}_{\AL,t}$. We multiply by $t^{-1/2}$, set it to zero before taking the absolute value on both sides: this yields the following relationship:
$$\left|t^{-1/2}\int_0^t Y^{(j)}_s dW^{(i)}_s + \left(t^{1/2}\vectorise(\widehat{\rmQ}_{\AL,t})^\top \cdot (t^{-1}\rmK_t) \otimes \Si^{-1}\right)_{d(j-1)+i}\right|  =  \lambda t^{(\gamma+1)/2}\cdot |t^{1/2}\widehat{\rmQ}_{t}|^{-\gamma}_{ij}.$$
On the right hand side, as mentioned above, we have $t^{1/2}(\widehat{\rmQ}_{\AL,t})_{|\widebar{\gG}_0} = o_p(1)$. Since $\lambda t^{(\gamma+1)/2}= O(1)$ and $\gamma > 0$, this side diverges to $\infty$ in probability as $t \rightarrow \infty$. The first term on the left-hand side is a martingale and is normally-distributed as $t\rightarrow \infty$. By the asymptotic normality of $t^{1/2}(\widehat{\rmQ}_{\AL,t})_{|\widebar{\gG}_0}$ and the fact that $t^{-1}\rmK_t \xrightarrow{\ p \ }\rmK_\infty$, the left-hand side is the absolute value of the sum of two Gaussian random variables  whose probability to be larger or equal to the right-hand side (which diverges to $\infty$) tends to zero. The right-hand side was computed under the assumption that $(\widehat{\rmQ}_{\AL,t})_{ij} \neq 0$, and a bound of the probability of that event happening is zero, hence we have that $\proba\left((\widehat{\rmQ}_{\AL,t})_{ij} = 0 \right) \rightarrow 1$ if $(i,j) \in \widebar{\gG}_0$. This shows the second and last property.}
\end{proof}

\subsection{Proof of Proposition \ref{proposition:sigma-mma}}
\label{proof:proposition:sigma-mma}
\begin{proof}[Proof of Proposition \ref{proposition:sigma-mma}]
	We define $\pi_{\VVV} :=\delta_{\VVV}$ a probability measure on $\Ss^{++}$ and there exists an $\Ss^+$-valued \Levy basis $\Lambda_{\LLL}$ on $\Ss^{++}\times \R$ corresponding to the characteristic quadruplet $(\boldsymbol{\gamma}_{\LLL}, \boldsymbol{0}, \nu_{\LLL}, \pi_{\VVV})$. We denote by $\vectorise(\Lambda_{\LLL}) = \left\{\vectorise(\Lambda(B)) : B \in \mathcal{B}_b(\Ss^{++}\times \R)\right\}$ the $\R^{d^2}$-valued \Levy basis corresponding to $\Lambda_{\LLL}$ and we rewrite $(\Si_t, \ t \in \R)$ as follows:
 	$$\vectorise(\Si_t) \overset{d}{=} \int_{\Ss^{+}}\int_{\R}\indicator_{[0,\infty)}(t-s)e^{-(t-s)(\boldsymbol{A}\otimes \Id + \Id \otimes \boldsymbol{A})}\vectorise(\Lambda_{\LLL}(d\boldsymbol{A}, ds)),$$
 	 which can also be interpreted as a $d^2$-dimensional \levy-driven MMA process which is mixing \citep[Theorem 3.5]{fuchs2013mixing}.
\end{proof}

\subsection{Proof of Proposition \ref{proposition:f1-non-degenerate}}
\label{proof:proposition:f1-non-degenerate}
\begin{proof}[Proof of Proposition \ref{proposition:f1-non-degenerate}]
Recall that
$$\vectorise\left(\int_{a}^b e^{-(b-s)\Q}\Si_s e^{-(b-s)\Q^\top}ds\right) = \int_a^b e^{-(b-s)(\rmQ\otimes \Id + \Id \otimes \Q)}\vectorise(\Si_s)ds.$$
Using the stationarity of $(\Si_t)$ and Equation \eqref{eq:integrated-variance-lebesgue}, by Fubini's theorem we have

\begin{align*}
	\int_a^b &e^{-(b-s)(\rmQ\otimes \Id + \Id \otimes \Q)}\vectorise(\Si_s)ds\\
	 &= \int_a^b e^{-(b-s)(\rmQ\otimes \Id + \Id \otimes \Q^\top)}\int_{-\infty}^s e^{-(s-u)(\VVV\otimes \Id + \Id \otimes \VVV^\top)}\vectorise(d\LLL_u) ds\\
	 &= \int_a^b e^{-(b-s)(\rmQ\otimes \Id + \Id \otimes \Q^\top)}\int_{-\infty}^0 e^{-v(\VVV\otimes \Id + \Id \otimes \VVV^\top)}\vectorise(d\LLL_v) ds\\
	 &= \int_{-\infty}^0 \int_a^b e^{-(b-s)(\rmQ\otimes \Id + \Id \otimes \Q^\top)} e^{-v(\VVV\otimes \Id + \Id \otimes \VVV^\top)} ds \times \vectorise(d\LLL_v) \\
	 &= \boldsymbol{\rho}^{-1}(\rmQ)\left[\Id - e^{-(b-a)(\rmQ\otimes \Id + \Id \otimes \Q^\top)}\right] \int_{-\infty}^0 e^{-v(\VVV\otimes \Id + \Id \otimes \VVV^\top)} \vectorise(d\LLL_v).
\end{align*}
Using Equations \eqref{eq:levy-matrix-subordinator-moment-conditions} \& \eqref{eq:characteristic-function-matrix-subordinator}, we apply Corollary 4.1, \cite{basse2014stochastic}  to have that
$$\int_{a}^b e^{-(b-s)\Q}\Si_s e^{-(b-s)\Q^\top}ds < \infty, \quad a.s.$$
for any $a < b \in \R \cup\{\pm\infty\}$. Indeed, the first and second conditions are true by assumption and definition of the subordinator $\LLL$ whilst the third condition is true if we consider equivalently the truncation functions  $\tau(\boldsymbol{z}) := \mathbb{I}_{\{\boldsymbol{x} \in \R^d : \|\boldsymbol{x}\| \leq 1\}}(\boldsymbol{z})$ on $\R^d$ and $\tau_d(\boldsymbol{Z}) := \mathbb{I}_{\{\boldsymbol{X} \in \MdR : \|\boldsymbol{X}\| \leq 1\}}(\boldsymbol{Z})$ on $\MdR$ instead of the truncation function $\widetilde{\tau}(\boldsymbol{x}) \equiv 0$ on both $\R^d$ and $\MdR$ as given in Part 1, \cite{barndorff2015change} and used in Section \ref{section:stoch-vol-characteristic-function}.
By the independence between $\LLL$ and $\WW$, we obtain that
$$\E\left(e^{i\boldsymbol{u}^\top \mathbb{F}^{(1)}_{ab}} \Big| \sigma\left(\{\Si_s, \ s \in [a,b]\}\right)\right) = \exp\left\{-\frac{1}{2}\boldsymbol{u}^\top\left(\int_{a}^b e^{-(b-s)\Q}\Si_s e^{-(b-s)\Q^\top}ds\right)\boldsymbol{u}\right\},$$
where $\sigma(\{\Si_s, \ s \in [a,b]\})$ is the $\sigma$-algebra generated by $(\Si_s, \ a \leq s \leq b)$. Hence, for some $c_1 > 0$ small enough, 
 we have for $\|u\| \leq c_1$ 
$$\E\left(e^{i\boldsymbol{u}^\top \mathbb{F}^{(1)}_{ab}} \Big| \sigma\left(\{\Si_s, \ s \in [a,b]\}\right)\right) \leq 1 -\frac{1}{2}\boldsymbol{u}^\top\left(\int_{a}^b e^{-(b-s)\Q}\Si_s e^{-(b-s)\Q^\top}ds\right)\boldsymbol{u} \quad a.s.$$
Let $c_2 :=\frac{1}{2} \min_{i,j} \left(\int_{a}^b e^{-(b-s)\Q}\E\left(\Si_1\right) e^{-(b-s)\Q^\top}ds\right)_{ij}$ and then, by Fubini's theorem, the distribution of $\mathbb{F}^{(1)}_{ab}$ is non-degenerate since
 $$\E\left(e^{i\boldsymbol{u}^\top \mathbb{F}^{(1)}_{ab}}\right)  \leq 1 -c_2\|\boldsymbol{u}\|^2, \quad \text{for any } \|\boldsymbol{u}\| \leq c_1,$$
for some $c_2 > 0$ according to Proposition 24.19, \cite{sato1999levy}.
\end{proof}

\subsection{Proof of Proposition \ref{proposition:f1-statio}}
\label{proposition:f1-statio:proof}
\begin{proof}[Proof of Proposition \ref{proposition:f1-statio}]
	Let $k \in \N$, $0 \leq t_1 < \dots < t_k$ and $h > 0$. Then
	$$(\mathbb{F}^{(1)}_{t_1+h},\dots, \mathbb{F}^{(1)}_{t_k+h}) \overset{d}{=} \left(\int_{-\infty}^{t_i+h}e^{-(t_i+h-{s_i})\Q}\Si_{{s_i}}^{1/2}d\WW_{s_i}, \ i \in \{1,\dots,k\}\right),$$
and defining $u_i = s_i - (t_i+h)+t_1$ for all $i \in \{1,\dots,k\}$ yields
$$(\mathbb{F}^{(1)}_{t_1+h},\dots, \mathbb{F}^{(1)}_{t_k+h}) \overset{d}{=} \left(\int_{-\infty}^{t_1}e^{-(t_1-u_i)\Q}\Si_{u_i+h +t_i-t_1}^{1/2}d\WW_{u_i +t_i-t_1 +h}, \ i \in \{1,\dots,k\}\right).$$
By strict stationarity of $(\Si_t)$, we have
$$(\Si_{u_1+h},\dots, \Si_{u_k+t_k-t_1+h}) \overset{d}{=} (\Si_{u_1},\dots, \Si_{u_k+t_k-t_1}),$$
and we conclude that
$$(\mathbb{F}^{(1)}_{t_1+h},\dots, \mathbb{F}^{(1)}_{t_k+h}) \overset{d}{=} (\YY_{t_1},\dots, \YY_{t_k}),$$
and this result on finite-dimensional distributions gives the strict stationarity of $\mathbb{F}^{(1)}$.
\end{proof}

\subsection{Proof of Proposition \ref{proposition:si-and-j2-mixing}}
\label{proposition:si-and-j2-mixing:proof}
 \begin{proof}[Proof of Proposition \ref{proposition:si-and-j2-mixing}]
 	We define $\pi_\rmQ := \delta_\Q$, a probability measures on $\Ss^{++}$, where $\delta$ is the Dirac delta distribution. According to Section \ref{section:levy-bases-and-mma-processes}, there exist an $\R^d$-valued \Levy basis $\Lambda_{\JJ}$  on $\Ss^{++} \times \R$ associated with the quadruplet $(T \times \boldsymbol{\gamma}_\JJ, \boldsymbol{0}, T \otimes \nu_\JJ, \pi_{\Q})$. By independence between $(\JJ_t)$ and $(T_t)$, we obtain  that
 	  	$$\mathbb{F}^{(2)}_{-\infty t} \overset{d}{=} \int_{\R^d}\int_{\R}\indicator_{[0,\infty)}(T_t-s)e^{-(T_t-s)\boldsymbol{A}}\vectorise(\Lambda_{\JJ}(d\boldsymbol{A}, ds)), \quad \forall t \in \R.$$
 	  	Both of those quantities can be interpreted as \levy-driven MMA process similarly to Equation (4.5), \cite{fuchs2013mixing}. Therefore, it is mixing by Theorem 3.5, \cite{fuchs2013mixing}.
 \end{proof}

\subsection{Proof of Proposition \ref{proposition:yt-mixing-stoch-vol}}
\label{proposition:yt-mixing-stoch-vol:proof}
 \begin{proof}[Proof of Proposition \ref{proposition:yt-mixing-stoch-vol}]
%
 	 
 	 Using Proposition \ref{proposition:sigma-mma} and the proof of Proposition \ref{proposition:f1-non-degenerate}, recall that $(\Si_t, \ t\in \R)$ is an MMA process hence mixing \citep[Theorem 2.5]{fuchs2013mixing} Theorem such that for all $t \in \R\cup \{\infty\}$
\begin{equation*}
	\int_{-\infty}^t e^{-(t-s)\Q} \Si_s e^{-(t-s)\Q^\top}ds < \infty \quad a.s.,
\end{equation*}
as a Lebesgue integral of $(\Si_t)$ $\omega$-wise (see Equation \eqref{eq:integrated-variance-lebesgue}). Also, recall that  $(\mathbb{F}^{(1)}_{-\infty t}, \ t \in \R)$ is stationary by Proposition \ref{proposition:f1-statio}.
From the integrability of $(\Si_t)$, we have that $(\mathbb{F}^{(1)}_{-\infty t}, \ t \in \R)$ is a centred Gaussian process where its autocovariance function $\Acov(h) := \E\left(\mathbb{F}^{(1)}_{-\infty t} \mathbb{F}^{(1)}_{-\infty t+h}\right)$ for $h \geq 0$ (and some $t \in \R$) is given by
\begin{align*}
\Acov(h) &= \E\left[\int_{-\infty}^{t} e^{-(t-u)\Q}\Si_u^{1/2}d\WW_u\left(\int_{-\infty}^{t+h} e^{-(t+h-v)\Q}\Si_v^{1/2}d\WW_v\right)^\top\right]\\
&= \left[\int_{-\infty}^t e^{-(t-u)\Q} \E\left(\Si_u\right)e^{-(t-u)\Q^\top} \right] e^{-h\Q^\top}, \quad \text{by Fubini's theorem,}\\
&=O(e^{-h\Q^\top}) = o(1), \quad \text{as $h\rightarrow \infty$ since $\E(\Si_t)= -\rho^{-1}(\E\left(\LLL_1\right) )< \infty$,}
\end{align*}
where we recall that $\rho(\boldsymbol{X}) = \VVV \boldsymbol{X} + \boldsymbol{X} \VVV^\top$.

According to Sections 3.7 \& 4.12, \cite{dym2008gaussian}, given that $\mathbb{F}^{(1)}$ is stationary and a centred Gaussian process such that its autocovariance $\Acov$ is continuous and such that $\Acov(h)=o(1)$ as $h\rightarrow \infty$, we conclude that $(\mathbb{F}^{(1)}_{-\infty t})$ is mixing. 
Therefore, by Proposition \ref{proposition:si-and-j2-mixing}, $(\YY_t, \ t \in \R)$ is the sum of two mixing processes hence itself mixing and ergodic.

 \end{proof}

\bibliographystyle{agsm}
\bibliography{grou-cts}

\end{document}